\documentclass[12pt]{amsart}
\usepackage{amsmath,amsfonts,amssymb,amsthm,amstext,pgf,graphicx,hyperref,verbatim,lmodern,textcomp,color,young,tikz}
\usetikzlibrary{decorations}
\usepackage[mathscr]{euscript}
\usetikzlibrary{decorations.markings}
\usetikzlibrary{arrows}
\usepackage{float}
\theoremstyle{plain}
\newtheorem{theorem}{Theorem}[section]

\newtheorem{corollary}{Corollary}

\theoremstyle{definition}
\newtheorem{defn}{Definition}[section]

\title{}
\begin{document}
\title [Combinatorialization of identities,  general  Linear Recurrences  ]{Combinatorialization of Sury and McLaughlin identities,  general  Linear Recurrences  IN A UNIFIED approach } 
\author[Sudip Bera]{Sudip Bera}
\address[Sudip Bera]{Department of Mathematics, Indian Institute of Science, Bangalore 560 012}
\email{sudipbera@iisc.ac.in}
\keywords{combinatorial proof; determinants; digraphs;   symmetric functions; words }
\subjclass[2010]{05A19; 05A05; 05C30; 05C38}
\maketitle
\begin{abstract} 
In this article we provide with combinatorial proofs of some recent identities due to Sury and McLaughlin. We show that, the solution of a general linear recurrence with constant coefficients can be interpreted as a determinant of a matrix. Also, we derive a determinantal expression of  Fibonacci and Lucas numbers. We prove Binet's formula for Fibonacci and Lucas numbers in a purely combinatorial way and in course of doing so, we find a  determinantal identity, which we think to be new. 
\end{abstract}
\section{Introduction}
This paper contributes to the study of combinatorial proofs of some recent identities due to Sury and McLaughlin, determinantal formulas of general linear recurrence with constant coefficients in a unified way. Combinatorial proofs give more insight into  ``why" the result is true rather than ``how" \cite{18, 14, SSDISC, 11}. 

It is not always possible to find a closed form expression for an arbitrary term in the sequence of a recurrence relation with initial conditions. Many of our favorite number sequences, such as Fibonacci numbers and their generalizations, Lucas numbers, are precisely these. Each has beautiful combinatorial interpretations using tilling of a board \cite{24}.  Fibonacci and Lucas numbers are defined by a second order linear recurrence with coefficients of $1$ with special initial conditions. There are many different proofs of Binet's formula for Fibonacci and Lucas numbers \cite{ 25, 29}. In particular, in \cite{29}, a combinatorial proof using a random tiling of an infinite board with squares and dominoes can be used to explain Binet's formula and its generalization for arbitrary initial conditions. In this paper, we prove that the solution of a general linear recurrence with constant coefficient can be expressed as a determinant in a purely combinatorial way and consequently we prove the Binet's formula regarding Fibonacci and Lucas numbers. Our combinatorial approach also yields some recent identities due to Sury and McLaughlin.   

Let us briefly summarize the content of this paper. In Section $2,$ we  give a bijective proof of an identity regarding elementary and homogeneous symmetric polynomials and as a corollary we give bijective proof of some recent identities due to Sury and McLaughlin. In Section $3,$ we prove that the solution of a general linear recurrence with constant coefficients can be interpreted as a determinant of some matrix. We prove Binet's formula for Fibonacci numbers in a combinatorial way. In Section $4,$ we derive a new combinatorial identity and as a corollary, we get Binet's formula for Lucas numbers.  
\section{ combinatorial explanation of Sury and McLaughlin  identities}
In this section, we give combinatorial proof of some recent identities due to Sury and McLaughlin. Before getting into that, we develop some necessary background.  For details we refer \cite{16,21}. Let $A=(a_{ij})_{n\times n}$ be a  matrix. Now we associate a weighted digraph $D(A),$ (with $A$) whose vertex set is $[n]=\{1, 2, \cdots, n\}$ and for each ordered pair $(i, j),$ there is an edge directed from $i$ to $j$ with weight $a_{ij}.$ A \emph{linear subdigraph} (LSD) $L,$ of $D(A)$ is a spanning collection of pairwise vertex-disjoint cycles. A loop around a single vertex is also considered to be a cycle. The \emph{weight} of a linear subdigraph $L,$ written as $w(L),$ is the product of the weights of all its cycles. The weight of a cycle is the product of weights of all its edges. The \emph{length} of a cycle is the number of edges  present in that cycle. The number of cycles contained in $L$ is denoted by $c(L).$ Now the cycle-decomposition of permutations yields the following description of $\det(A),$ namely  $\det(A)=\sum\limits_{L}(-1)^{n-c(L)}w(L),$ where the summation runs over all linear subdigraphs $L$ of $D(A).$ A \emph{partition} $\lambda$ of a positive integer $m$ is a weakly decreasing finite sequence $(\lambda_1,\cdots,\lambda_r)$ of non-negative integers such that $\sum\limits_{i=1}^{r}\lambda_i=m.$
We denote $\lambda\vdash m$ to mean $\lambda$ is a partition of $m$. For $n$ variables $x_1,\cdots, x_n$, the elementary symmetric polynomial $e_k(x_1,\cdots, x_n),$ (in short, $e_k$) of degree $k \geq 0,$ is defined as \[e_k(x_1,\cdots, x_n):=\sum\limits_{1\leq j_1<j_2\cdots<j_k\leq n}x_{j_1}x_{j_2}\cdots x_{j_k}, \text{ and } e_0=1.\] The complete homogeneous symmetric polynomial $h_k(x_1,\cdots, x_n),$ (in short, $h_k$) of degree $k\geq 0,$ in $n$ variables $x_1, \cdots, x_n,$  is the sum of all monomials of total degree $k.$  Formally,
\[h_k(x_1,\cdots, x_n):=\sum\limits_{1\leq j_1\leq j_2\cdots \leq j_k\leq n}x_{j_1}x_{j_2}\cdots x_{j_k}, \text{ and } h_0=1.\] Let $\lambda = (\lambda_1, \cdots, \lambda_{\ell})$ be a partition and $\delta=(\ell-1, \ell-2, \cdots, 1, 0),$  where $\lambda_1\geq \cdots\geq \lambda_{\ell}$ and each $\lambda_j$ is a non-negative integer. Then the functions 
\[a_{\lambda+\delta}(x_1, \cdots, x_{\ell})=\det \left(
\begin{array}{cccc}
x_1^{\lambda_1+\ell-1} & x_2^{\lambda_1+\ell-1}& \cdots & x_{\ell}^{\lambda_1+\ell-1}\\ 
x_1^{\lambda_2+\ell-2} & x_2^{\lambda_2+\ell-2}& \cdots & x_{\ell}^{\lambda_2+\ell-2} \\
\vdots & \vdots & \ddots &\vdots \\
x_1^{\lambda_{\ell}} & x_2^{\lambda_{\ell}}& \cdots & x_{\ell}^{\lambda_{\ell}}
\end{array}
\right)\] are alternating polynomials. The Schur polynomials are defined as the ratio \[\frac{a_{\lambda+\delta}(x_1, \cdots, x_{\ell})}{a_{\delta}(x_1, \cdots, x_{\ell})}.\]

Let $e_1, \cdots, e_m$ be elementary symmetric polynomials in the variables $x_1, \cdots, x_n.$  We consider a $m\times m$ matrix
\begin{equation}\label{ele hom matrix}
E(e_1, \cdots, e_m)= \left(
\begin{array}{cccccc}
e_1 & e_2 & e_3 & \cdots & e_{m-1} & e_m\\
1& e_1 & e_2 & \cdots & e_{m-2} & e_{m-1} \\
0 & 1 & e_1 &\cdots  &  e_{m-3} & e_{m-2} \\
\vdots & \vdots & \vdots &\ddots & \vdots & \vdots \\
0 & 0 & 0 & \cdots & e_1 & e_2\\ 
0 & 0 & 0 & \cdots & 1 & e_1\\
\end{array}
\right).
\end{equation}
\begin{theorem}[Theorem $4.5.1$ \cite{21}]\label{ele and hom relation}
Let $e_1, \cdots, e_m  $ be elementary symmetric polynomials in the variables  $x_1, x_2, \cdots, x_n.$ Let $E(e_1, \cdots, e_m)$ be an $m\times m$ matrix defined as \eqref{ele hom matrix}. Then $\det(E(e_1, \cdots, e_m))=h_m(x_1, \cdots, x_n).$
\end{theorem} 
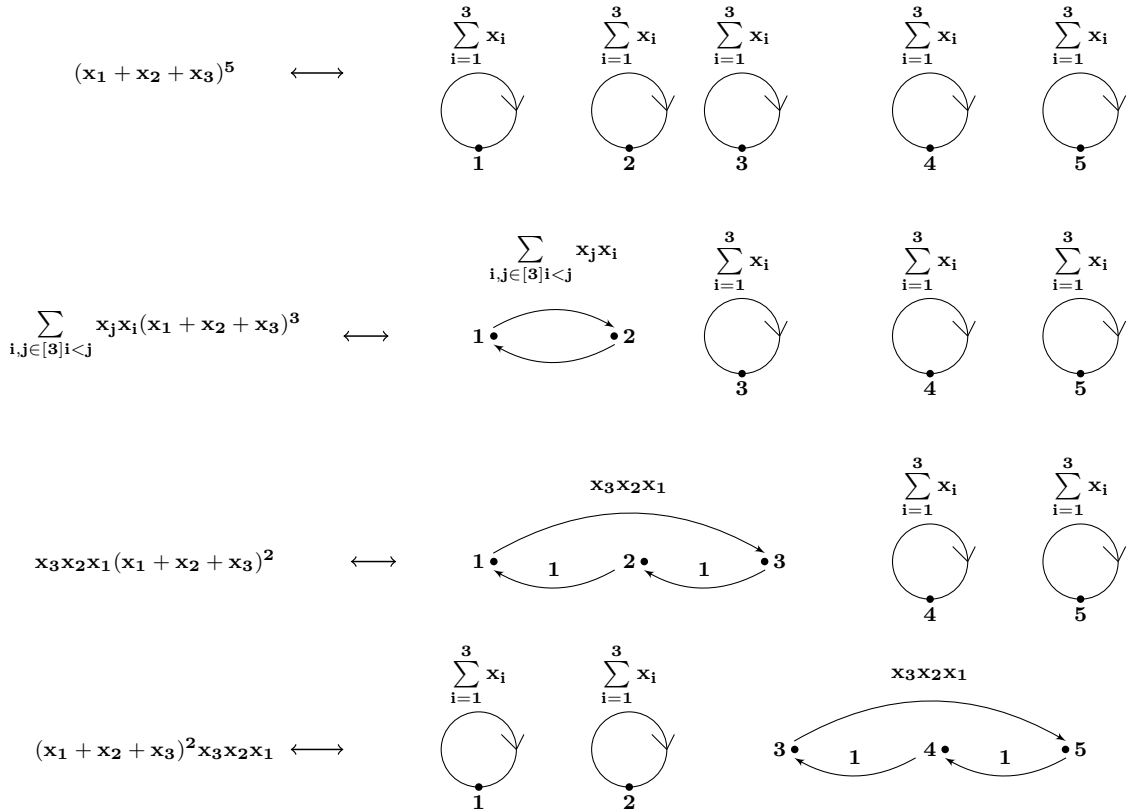
\begin{figure}[ht!]
	\tiny
	\tikzstyle{ver}=[]
	\tikzstyle{vert}=[circle, draw, fill=black!100, inner sep=0pt, minimum width=4pt]
	\tikzstyle{vertex}=[circle, draw, fill=black!.5, inner sep=0pt, minimum width=4pt]
	\tikzstyle{edge} = [draw,thick,-]
	\tikzstyle{node_style} = [circle,draw=blue,fill=blue!20!,font=\sffamily\Large\bfseries]
	\centering
	\tikzset{->-/.style={decoration={
				markings,
				mark=at position #1 with {\arrow{>}}},postaction={decorate}}}
\begin{tikzpicture}[scale=1]
	\tikzstyle{edge_style} = [draw=black, line width=2mm, ]
	\tikzstyle{node_style} = [draw=blue,fill=blue!00!,font=\sffamily\Large\bfseries]
	 \tikzset{
	 	LabelStyle/.style = { rectangle, rounded corners, draw,
	 		minimum width = 2em, fill = yellow!50,
	 		text = red, font = \bfseries },
	 	VertexStyle/.append style = { inner sep=5pt,
	 		font = \Large\bfseries},
	 	EdgeStyle/.append style = {->, bend left} }
	 \tikzset{vertex/.style = {shape=circle,draw,minimum size=1.5em}}
	 \tikzset{edge/.style = {->,> = latex'}}
	%
\node (B1) at (-2.3,0) {$\bf{(x_1+x_2+x_3)^5}$}; 
 \draw[<->, line width=.2 mm] (-.5,0) -- (.2,0);
 \draw[](2,-.5) circle (.5);
 \draw[] (2.5,-.5) -- (2.6,-.3);
 \draw[] (2.5,-.5) -- (2.3,-.3);
 \node (B1) at (2,-1.2)   {$\bf{1}$};
 \fill[black!100!](2,-1) circle (.05);
 \node (B1) at (2,.5)   {$\bf{\sum\limits_{i=1}^3 x_i}$};
 \draw[](4,-.5) circle (.5);
 \node (B1) at (4,-1.2)   {$\bf{2}$};
 \fill[black!100!](4,-1) circle (.05);
 \node (B1) at (4, .5) {$\bf{\sum\limits_{i=1}^3 x_i}$};
 \draw[] (4.5,-.5) -- (4.6,-.3);
 \draw[] (4.5,-.5) -- (4.3,-.3);
 \draw[](5.5,-.5) circle (.5);
 \draw[] (6,-.5) -- (6.1,-.3);
 \draw[] (6,-.5) -- (5.8,-.3);
 \node (B1) at (5.5, -1.2)   {$\bf{3}$};
 %
 \fill[black!100!](5.5,-1) circle (.05);
 \node (B1) at (5.5,.5)   {$\bf{\sum\limits_{i=1}^3 x_i}$};
 \draw[](8,-.5) circle (.5);
 \draw[] (8.5,-.5) -- (8.6,-.3);
 \draw[] (8.5,-.5) -- (8.3,-.3);
 \node (B1) at (8,-1.2)   {$\bf{4}$};
 \fill[black!100!](8,-1) circle (.05);
 \node (B1) at (8,.5)   {$\bf{\sum\limits_{i=1}^3 x_i}$};
 \draw[](10,-.5) circle (.5);
 \node (B1) at (10,-1.2)   {$\bf{5}$};
 \fill[black!100!](10,-1) circle (.05);
 \node (B1) at (10, .5) {$\bf{\sum\limits_{i=1}^3 x_i}$};
 \draw[] (10.5,-.5) -- (10.6,-.3);
 \draw[] (10.5,-.5) -- (10.3,-.3);
\node (B1) at (-2.3,-3.5)   {$\bf{\sum\limits_{ i, j\in [3]  i<j } x_jx_i(x_1+x_2+x_3)^3}$}; 
\draw[<->, line width=.2 mm] (0.2,-3.5) -- (.8,-3.5);
\node[] (a) at  (2,-3.5) {$\bf{1}$};
\node[] (b) at  (4,-3.5) {$\bf{2}$};
\fill[black!100!](2.2,-3.5) circle (.05);
\fill[black!100!](3.8,-3.5) circle (.05);
\draw[edge] (a)  to[bend left] (b);
\draw[edge] (b)  to[bend left] (a);
\node (B1) at (3, -2.5)   {$\bf{\sum\limits_{ i, j\in [3]  i<j } x_jx_i}$};
\draw[](5.5,-3.5) circle (.5);
\draw[] (6,-3.5) -- (6.1,-3.3);
\draw[] (6,-3.5) -- (5.8,-3.3);
\node (B1) at (5.5, -4.2)   {$\bf{3}$};
%
\fill[black!100!](5.5,-4) circle (.05);
\node (B1) at (5.5,-2.5)   {$\bf{\sum\limits_{i=1}^3 x_i}$};
\draw[](8,-3.5) circle (.5);
\draw[] (8.5,-3.5) -- (8.6,-3.3);
\draw[] (8.5,-3.5) -- (8.3,-3.3);
\node (B1) at (8,-4.2)   {$\bf{4}$};
\fill[black!100!](8,-4) circle (.05);
\node (B1) at (8,-2.5)   {$\bf{\sum\limits_{i=1}^3 x_i}$};
\draw[](10,-3.5) circle (.5);
\node (B1) at (10,-4.2)   {$\bf{5}$};
\fill[black!100!](10,-4) circle (.05);
\node (B1) at (10, -2.5) {$\bf{\sum\limits_{i=1}^3 x_i}$};
\draw[] (10.5,-3.5) -- (10.6,-3.3);
\draw[] (10.5,-3.5) -- (10.3,-3.3);
\node (B1) at (-2.3,-6.5)   {$\bf{x_3x_2x_1(x_1+x_2+x_3)^2}$};
\draw[<->, line width=.2 mm] (0.3,-6.5) -- (.9,-6.5);
\node[] (c) at  (2,-6.5) {$\bf{1}$};
\node[] (d) at  (4,-6.5) {$\bf{2}$};
\node[] (e) at  (6,-6.5) {$\bf{3}$};
\fill[black!100!](2.2,-6.5) circle (.05);
\fill[black!100!](4.2,-6.5) circle (.05);
\fill[black!100!](5.8,-6.5) circle (.05);
\draw[edge] (c)  to[bend left] (e);
\draw[edge] (e)  to[bend left] (d);
\draw[edge] (d)  to[bend left] (c);
 \node (B1) at (4, -5.5)   {$\bf{x_3x_2x_1}$}; 
\node (B1) at (3, -6.6)   {$\bf{1}$}; 
\node (B1) at (5, -6.6)   {$\bf{1}$}; 
\draw[](8, -6.5) circle (.5);
\node (B1) at (8,-5.5)   {$\bf{\sum\limits_{i=1}^3 x_i}$};
\draw[] (8.5,-6.5) -- (8.6,-6.3);
\draw[] (8.5,-6.5) -- (8.3,-6.3);
\node (B1) at (8,-7.2)   {$\bf{4}$};
\fill[black!100!](8,-7) circle (.05);
%
\draw[](10,-6.5) circle (.5);
\draw[] (10.5, -6.5) -- (10.6,-6.3);
\draw[] (10.5,-6.5) -- (10.3,-6.3);
\node (B1) at (10,-7.2)   {$\bf{5}$};
\fill[black!100!](10,-7) circle (.05);
\node (B1) at (10, -5.5) {$\bf{\sum\limits_{i=1}^3 x_i}$};
\node (B1) at (-2.3,-9) {$\bf{(x_1+x_2+x_3)^2x_3x_2x_1}$}; 
\draw[<->, line width=.2 mm] (-.5,-9) -- (.2,-9);
\draw[](2,-9) circle (.5);
\draw[] (2.5,-9) -- (2.6,-8.8);
\draw[] (2.5,-9) -- (2.3,-8.8);
\node (B1) at (2,-9.7)   {$\bf{1}$};
\fill[black!100!](2,-9.5) circle (.05);
\node (B1) at (2,-8)   {$\bf{\sum\limits_{i=1}^3 x_i}$};
\draw[](4,-9) circle (.5);
\node (B1) at (4,-9.7)   {$\bf{2}$};
\fill[black!100!](4,-9.5) circle (.05);
\node (B1) at (4, -8) {$\bf{\sum\limits_{i=1}^3 x_i}$};
\draw[] (4.5,-9) -- (4.6,-8.8);
\draw[] (4.5,-9) -- (4.3,-8.8);
\node[] (c) at  (6,-9) {$\bf{3}$};
\node[] (d) at  (8,-9) {$\bf{4}$};
\node[] (e) at  (10,-9) {$\bf{5}$};
\fill[black!100!](6.2,-9) circle (.05);
\fill[black!100!](8.2,-9) circle (.05);
\fill[black!100!](9.8,-9) circle (.05);
\draw[edge] (c)  to[bend left] (e);
\draw[edge] (e)  to[bend left] (d);
\draw[edge] (d)  to[bend left] (c);
\node (B1) at (8, -8)   {$\bf{x_3x_2x_1}$}; 
\node (B1) at (7, -9.1)   {$\bf{1}$}; 
\node (B1) at (9, -9.1)   {$\bf{1}$};
%
\end{tikzpicture}
\caption{The numbers appearing on the edges in the above diagram are the weights of the corresponding edges. The left hand side of the figure contains some terms of PIE expression.}
\label{fig:f1}	
\end{figure}
\begin{proof}
 We prove the theorem for the case $m=5, n=3$ (similarly  we can prove the general case). Think $W(x_1, x_2, x_3)=\{x_1, x_2,  x_3\}$ to be the set of letters. The \emph{free monoid} $W(x_1, x_2, x_3)^*$ is the set of all finite sequences (including the empty sequence, denoted by $1$) of elements of $W(x_1, x_2, x_3),$ usually called \emph{linear words}, with the operation of concatenation. Construct an algebra  from $W(x_1, x_2, x_3)^*$ by taking formal sum of elements of $W(x_1, x_2, x_3)$ with coefficient in $\mathbb{Z},$ extending the multiplication by usual distributivity. For example, in this algebra, $(x_{1}+x_{2}+x_{3})(x_{1}+x_{2}+x_{3})$ [written as $(x_{1}+x_{2}+x_{3})^2]= x_{1}x_{1}+x_{1}x_{2}+x_{1}x_{3}+x_{2}x_{1}+x_{2}x_{2}+x_{2}x_{3}+x_{3}x_{1}+x_{3}x_{2}+x_{3}x_{3}; (1+x_{3})x_{2}=x_{2}+x_{3}x_{2}$ etc.  
 
 Let $W_L$ be the sum of all linear words in $W(x_1, x_2, x_3)^*$ of length $5,$ where the letter $x_2$  does not occur just before the letter $x_1$ and the letter $x_3$ does not occur just before the letters $x_1$ and $x_2$ i.e. $x_2x_1$, $x_3x_1$ and $x_3x_2$ does not occur in the words as a consecutive pair. So, \[W_L=x_1x_1x_1x_1x_1+x_1x_1x_1x_1x_2+\cdots+x_2x_2x_2x_2x_2+x_1x_1x_1x_1x_3+\cdots+x_3x_3x_3x_3x_3.\] Now we compute $W_L$ by using the \emph{Principle of Inclusion and Exclusion (PIE)} rule.
 
  The sum of all possible words of length $5$ is $(x_1+x_2+x_3)^5.$ This can be written as  \[(x_1+x_2+x_3)(x_1+x_2+x_3)(x_1+x_2+x_3)(x_1+x_2+x_3)(x_1+x_2+x_3).\] The sum of all possible words of length $5,$ where there is an occurrence of $x_2x_1$  at least as the first two letters is $x_2x_1(x_1+x_2+x_3)^3.$ Similarly, the sum of all possible words of length $5$ where there is an occurrence of $x_3x_2$ and $x_3x_1$ at least as the  first two letters is   is $x_3x_2(x_1+x_2+x_3)^3$  and  $x_3x_1(x_1+x_2+x_3)^3$ respectively. Again the sum of all possible words of length $5,$ where there is an occurrence of $x_3x_2x_1$  at least as the first three letters is $x_3x_2x_1(x_1+x_2+x_3)^2.$  Proceeding  this way and using the PIE rule we get,
\begin{align*}
W_L=&(x_1+x_2+x_3)^5-x_2x_1(x_1+x_2+x_3)^3-x_3x_1(x_1+x_2+x_3)^3-x_3x_2(x_1+x_2+x_3)^3-\\&(x_1+x_2+x_3)x_2x_1(x_1+x_2+x_3)^2 -(x_1+x_2+x_3)x_3x_1(x_1+x_2+x_3)^2-\\&(x_1+x_2+x_3)x_3x_2(x_1+x_2+x_3)^2+\cdots+x_3x_2x_1(x_1+x_2+x_3)^2 +\\&(x_1+x_2+x_3)x_3x_2x_1(x_1+x_2+x_3)+(x_1+x_2+x_3)^2x_3x_2x_1.
\end{align*}
Now the terms appearing in the above PIE expression of $W_L$ are $(x_1+x_2+x_3)^5, x_2x_1(x_1+x_2+x_3)^3, x_3x_1(x_1+x_2+x_3)^3, x_3x_2(x_1+x_2+x_3)^3, \cdots,  x_3x_2x_1(x_1+x_2+x_3)^2, (x_1+x_2+x_3)x_3x_2x_1(x_1+x_2+x_3), (x_1+x_2+x_3)^2x_3x_2x_1 $ and signs of the  corresponding terms are $+, -, -, -, \cdots, +, +, +$ respectively. We define the weight of a word  $w$ in $W(x_1, x_2, x_3)^*$ is the homomorphic image $f(w),$ where $f: W(x_1, x_2, x_3)^*\rightarrow \mathbb{Z}[x_1, x_2, x_3]$ is a ring homomorphism, with $f(x_1)=x_1$, $f(x_2)=x_2$ and $f(x_3)=x_3.$ Now Figure \ref{fig:f1} illustrates a bijection between terms appear in PIE expression and linear subdigraphs of $D(E(e_1, \cdots, e_m)),$ with equal weights and signs on both sides. What we really mean by this, is that the weight of the linear subdigraph in the right hand side of Figure \ref{fig:f1} is the image of the corresponding word in the left hand side, under the ring homomorphism $f: W(x_1, x_2, x_3)^*\rightarrow \mathbb{Z}[x_1, x_2, x_3].$
\end{proof}
 Using the same combinatorial model used in the above proof, we now re derive some interesting recent identities of Sury and McLaughlin \cite{26, 27}. It is to be noted that, the original proofs of those identities involve to some extent cumbersome algebraic manipulation, whereas our approach is purely combinatorial.   
\begin{corollary}[Theorem $1$ \cite{26}]\label{iden1}\label{coro 1}
Let $x_1, \cdots, x_k$ be independent variables. Let $e_1, \cdots, e_k$ denote the various elementary symmetric polynomials in the $x_i"s$ of degrees $1, 2, \cdots, k$ respectively. Then in the polynomial ring $K[x_1, \cdots, x_k], (K$ is a field of characteristic  $0$) for each positive integer $n,$ one has the identity
\begin{align*}
&\sum\limits_{r_1+r_2+\cdots+r_k=n, r_i\geq 0}x_1^{r_1}x_2^{r_2}\cdots x_k^{r_k}\\
&=\sum\limits_{2i_2+3i_3+\cdots+ki_k\leq n}c(i_2, \cdots, i_k, n)e_1^{n-2i_2-3i_3-\cdots-ki_k}\times (-e_2)^{i_2}\times  e_3^{i_3}\times \cdots\times  ((-1)^{k-1}e_k)^{i_k},
\end{align*} 
where \[c(i_2, \cdots, i_k, n)=\frac{(n-i_2-2i_3-\cdots-(k-1)i_k)!}{i_2!\cdots i_k!(n-2i_2-3i_3-\cdots-ki_k))!}.\]   	
\end{corollary}
\begin{proof}
We consider the $n\times n$ matrix  \[E(e_1,\cdots, e_k, 0\cdots, 0)=\left(
\begin{array}{ccccccccc}
e_1&e_2&e_3&\cdots&e_k&0&\cdots&0&0\\
1&e_1&e_2&\cdots&e_{k-1}&e_k&\cdots&0&0\\
0&1&e_1&\cdots&e_{k-2}&e_{k-1}&\cdots&0&0\\
\vdots&\vdots&\vdots&\ddots&\vdots&\vdots&\ddots&\vdots&\vdots\\
0&0&0&\cdots&e_1&e_2&\cdots&e_{n-k}&e_{n-k+1}\\
0&0&0&\cdots&1&e_1&\cdots&e_{n-k-1}&e_{n-k}\\
\vdots&\vdots&\vdots&\ddots&\vdots&\vdots&\ddots&\vdots&\vdots\\
0&0&0&\cdots&0&0&\cdots& e_1&e_2\\
0&0&0&\cdots&0&0&\cdots&1&e_1
\end{array}
\right).\] By Theorem \ref{ele and hom relation}, we get
\begin{equation}\label{hom eqn1}
 \det(E(e_1,\cdots, e_k, 0\cdots, 0))=\sum\limits_{r_1+r_2+\cdots+r_k=n, r_i\geq 0 }x_1^{r_1}x_2^{r_2}\cdots x_k^{r_k}. 
\end{equation}
 We know that
\begin{equation}\label{hom eqn2}
\det(E(e_1,\cdots, e_k, 0\cdots, 0))=\sum\limits_{L}(-1)^{n-c(L)}w(L),
 \end{equation}
 \text{ where the sumation runs over all LSD  $L$ of the digraph   } $D(E(e_1,\cdots, e_k, 0\cdots, 0)).$ Clearly, each LSD of $D(E(e_1,\cdots, e_k, 0\cdots, 0))$  contains cycles of length at most $k.$ Suppose $L$ be an arbitrary LSD, and $L$ contains $i_{t}$ many cycles of length $t (t=2, 3, \cdots, k).$ So $L$ contains $(n-2i_2-3i_3-\cdots-ki_k)$ many loops. Now $(-1)^{n-c(L)}=(-1)^{i_2+2i_3+3i_4+\cdots+(k-1)i_k}.$ Again, weight of the LSD, $L$ is \[e_1^{n-2i_2-3i_3-\cdots-ki_k}\times (e_2)^{i_2}\times e_3^{i_3}\times \cdots\times  (e_k)^{i_k}\] 
 and  \[c(i_2, \cdots, i_k, n)=\frac{(n-i_2-2i_3-\cdots-(k-1)i_k)!}{i_2!\cdots i_k!(n-2i_2-3i_3-\cdots-ki_k))!},\] is the number of LSD $L,$ containing $n-2i_2-3i_3-\cdots-ki_k$ many loops and $i_{t} (t=2, 3, \cdots, k)$ many cycles of length $t.$ Now putting the values of weight and sign of each LSD in Equation \eqref{hom eqn2}, we get 
 \begin{align*}\label{hom eqn3}
 \det(E(&e_1, \cdots, e_k, 0, \cdots, 0))\\=&\sum\limits_{2i_2+3i_3+\cdots+ki_k\leq n}c(i_2, \cdots, i_k, n)e_1^{n-2i_2-3i_3-\cdots-ki_k}\times (-e_2)^{i_2}\times  e_3^{i_3}\times \cdots\times  ((-1)^{k-1}e_k)^{i_k}.	
 \end{align*}
 Hence the identity.	   
 \end{proof}
 \begin{corollary}[Theorem $3$ \cite{26}]
Let $n$ be a positive integer and $x, y, z$ be indeterminates. Then 
\begin{align*}
&\sum\limits_{2i+3j\leq n}(-1)^i\binom{i+j}{j}\binom{n-i-2j}{i+j}(x+y+z)^{n-2i-3j}(xy+yz+zx)^i(xyz)^j\\
&=\frac{xy(x^{n+1}-y^{n+1})-xz(x^{n+1}-z^{n+1})+yz(y^{n+1}-z^{n+1})}{(x-y)(x-z)(y-z)}.
\end{align*}	
\end{corollary}
\begin{proof}
Suppose $\lambda=(n)$ be a partition of $n.$ Then it is clear that the right hand side of this identity is the Schur polynomial \[s_{\lambda}(x, y, z)=\frac{a_{\lambda+\delta}(x, y, z)}{a_{\delta}(x, y, z)}, \text{ where }\]
\begin{align*}
a_{\lambda+\delta}(x, y, z)=\det\left(
\begin{array}{ccc}
x^{n+2}&y^{n+2}&z^{n+2}\\
x&y&z\\
1&1&1
\end{array}
\right) \text{ and }
a_{\delta}(x, y, z)=\det\left(
\begin{array}{ccc}
x^{2}&y^{2}&z^{2}\\
x&y&z\\
1&1&1
\end{array}
\right).
\end{align*} 
Now for the partition $\lambda=(n), s_{\lambda}(x, y, z)$ is the complete homogeneous polynomial of degree $n$ with the variables $x, y, z.$ So by Theorem \ref{ele and hom relation}, \[h_n(x, y, z)=\det\left(
\begin{array}{cccccc}
e_1&e_2&e_3&\cdots&0&0\\
1&e_1&e_2&\cdots&0&0\\
0&1&e_1&\cdots&0&0\\
\vdots&\vdots&\vdots&\ddots&\vdots&\vdots\\
0&0&0&\cdots&e_1&e_2\\
0&0&0&\cdots&1&e_1
\end{array}
\right)=\sum\limits_{L}(-1)^{n-c(L)}w(L),\] where $e_1, e_2, e_3$ are elementary symmetric polynomials in the variables $x, y, z$ and order of the matrix is $n.$ Now applying same argument as in proof of Corollary \ref{coro 1}, we get the identity.
\end{proof}		
\begin{corollary}\label{cor3}
	Let $x, y$ be indeterminates. Then the following polynomial identity holds:  \[\sum\limits_{2i\leq n}(-1)^i \binom{n-i}{i}(x+y)^{n-2i}(xy)^i=x^n+x^{n-1}y+\cdots+xy^{n-1}+y^n\]  	
\end{corollary}
\begin{proof}
	Clearly the right hand side of the above identity is complete homogeneous polynomial $h_n(x, y)$ of degree $n.$  Now by Theorem \ref{ele and hom relation},
	\begin{align}\label{bienet}
	 \det\left(
	\begin{array}{cccccc}
	x+y &xy & 0 & \cdots & 0 & 0\\
	1& x+y & xy & \cdots & 0 & 0 \\
	0 & 1 & x+y &\cdots  & 0 & 0 \\
	\vdots & \vdots & \vdots &\ddots & \vdots & \vdots \\
	0 & 0 & 0 & \cdots & x+y & xy\\ 
	0 & 0 & 0 & \cdots & 1 & x+y\\
	\end{array}
	\right)=h_n(x, y),
	\end{align}
	 where the matrix is of order $n.$ Now  \[\det\left(
	\begin{array}{cccccc}
	x+y &xy & 0 & \cdots & 0 & 0\\
	1& x+y & xy & \cdots & 0 & 0 \\
	0 & 1 & x+y &\cdots  & 0 & 0 \\
	\vdots & \vdots & \vdots &\ddots & \vdots & \vdots \\
	0 & 0 & 0 & \cdots & x+y & xy\\ 
	0 & 0 & 0 & \cdots & 1 & x+y\\
	\end{array}
	\right)=\sum\limits_{L}(-1)^{n-c(L)}w(L),\] where the summation runs over all linear subdigraphs $L.$ Now proceeding same way as in proof of Corollary \ref{coro 1}, we get the identity.  
\end{proof}
 \section{determinantal interpretation of general linear recurrences }
 In this section, we prove that the solution of general linear recurrence with constant coefficients can be expressed as a determinant.  
\begin{defn}\label{gen liner recc}
 Let  $c_1, c_2,\cdots, c_r$ be real numbers. Let $u_0, u_1, u_2,\cdots, u_{r-1}$ be sequence of numbers, then for $n\geq 1,$ a $r$-th order linear recurrence is defined by   \[u_n=c_1u_{n-1}+c_2u_{n-2}+\cdots+c_ru_{n-r},\] with initial conditions $u_0=1$ and for $j\leq 0, u_j=0 $.
\end{defn}
There is a more or less well known  combinatorial interpretation of general linear recurrences (see  \cite{24}). In fact, $u_n$ is the sum of weights of all tillings of an $n$-board (a board of length $n$) with tiles of length  at most $r,$ where for $1\leq i\leq r,$ weight of each tile of length $i$ is  $c_i$ and the weight of a tilling is the product of weights of all tiles in that tilling. Now, we consider the $n$-ordered matrix
  \[C=\left(
  \begin{array}{ccccccccc}
  c_1 & -c_2 & c_3&\cdots & (-1)^{r+1}c_r & 0 & \cdots & 0&0 \\
  1& c_1 &-c_2& \cdots & (-1)^rc_{r-1} & (-1)^{r+1}c_r &\cdots& 0&0 \\
  0 & 1 & c_1& \cdots & (-1)^{r-1}c_{r-2}& (-1)^rc_{r-1} &\cdots& 0&0 \\
  \vdots & \vdots &\vdots& \ddots &\vdots & \vdots & \ddots&\vdots&\vdots \\
  0&0&0&\cdots&c_1&-c_2&\cdots&(-1)^{n-r+1}c_{n-r}&(-1)^{n-r+2}c_{n-r+1}\\
  0&0&0&\cdots&1&c_1&\cdots&(-1)^{n-r}c_{n-r-1}&(-1)^{n-r+1}c_{n-r}\\
  \vdots & \vdots &\vdots& \ddots &\vdots & \vdots & \ddots&\vdots&\vdots \\
  0 & 0 & 0&\cdots & 0 & 0 &\cdots &c_1 & -c_2\\
  0 & 0 & 0&\cdots & 0 & 0 &\cdots &1 & c_1\\
  \end{array}
  \right).\]
\begin{theorem}\label{gen lrr the}
Let $u_n (n\geq 1)$ be a general linear recurrence defined as \eqref{gen liner recc}. Then $u_n=\det(C).$	
\end{theorem}
\begin{proof}
To prove the theorem we use the combinatorial interpretation (stated above) of $u_n$ and the fact 	
 $\det(C)=\sum\limits_L(-1)^{n-c(L)}w(L),$ 
where the summation runs over all LSD, $L$ in  $D(C)$ (as in Section $2$). In fact, we show a sign and weight preserving bijection between tillings of $n$-board and linear subdigraphs of $D(C).$ Suppose $\tau$ be a tilling of the $n$-board and $\tau$ contains tiles of length at most $r.$ Let for each fix $i\in [r], \tau$ contains $t_i$ many tiles $\tau_{i1}, \cdots, \tau_{it_i}$ of length $i.$ For fix $j\in [t_i],$ let the tile $\tau_{ij}$ occupies the position $k_{ij}, (k_{ij}+1),\cdots (k_{ij}+i-1),$ where  $k_{ij} \in [n].$ (For example, the left hand side of Figure \ref{fig:fibbo} contains two tiles of lengths $2$ and $3,$ occupying the positions $1, 2$, and $3, 4, 5$ respectively). For this tilling we choose the LSD, $L_{\tau},$ such that $L_{\tau}$ contains $t_i$ many cycles $C_{i1}, \cdots, C_{it_i}$ of length $i.$ Moreover the cycle $C_{ij}$ (corresponding to the tile $\tau_{ij},$) contains the vertices $k_{ij}, (k_{ij}+1),\cdots (k_{ij}+i-1),$ and $k_{ij}\rightarrow(k_{ij}+i-1)\rightarrow(k_{ij}+i-2)\rightarrow \cdots\rightarrow(k_{ij}+1)\rightarrow k_{ij}$ (here $u\rightarrow v$ means a edge directed from $u$ to $v$). Clearly, this is a bijection. See Figure \ref{fig:fibbo} for an illustration.      
  \begin{figure}[ht!]
  	\tiny
  	\tikzstyle{ver}=[]
  	\tikzstyle{vert}=[circle, draw, fill=black!100, inner sep=0pt, minimum width=4pt]
  	\tikzstyle{vertex}=[circle, draw, fill=black!.5, inner sep=0pt, minimum width=4pt]
  	\tikzstyle{edge} = [draw,thick,-]
  	\tikzstyle{node_style} = [circle,draw=blue,fill=blue!20!,font=\sffamily\Large\bfseries]
  	\centering
  \begin{tikzpicture}[scale=1]
  	\tikzstyle{edge_style} = [draw=black, line width=2mm, ]
  	\tikzstyle{node_style} = [draw=blue,fill=blue!00!,font=\sffamily\Large\bfseries]
  	\tikzset{
  		LabelStyle/.style = { rectangle, rounded corners, draw,
  			minimum width = 2em, fill = yellow!50,
  			text = red, font = \bfseries },
  		VertexStyle/.append style = { inner sep=5pt,
  			font = \Large\bfseries},
  		EdgeStyle/.append style = {->, bend left} }
  	\node (a) at (7,.5)   {$\bf{1}$};
  	\node (b) at (15,.5)   {$\bf{5}$};
  	\node (c) at (13,.8)   {$\bf{4}$};
  	\node (d) at (11,.8)   {$\bf{3}$};
  	\node (e) at (9,.8)   {$\bf{2}$};
   \draw  (0,0) rectangle (5,1);
   \node (1) at (1,1.2)   {$\bf{c_2}$};
   \node (1) at (3.5,1.2)   {$\bf{c_3}$};
   \draw  (1,0) -- (1,1);
   \draw  (2,0) -- (2,1);
   \draw  (3,0) -- (3,1);
   \draw  (4,0) -- (4,1);
   \draw  [line width=.35 mm](.5,.5) -- (1.5,.5);
    \draw  [line width=.35 mm](2.5,.5) -- (4.5,.5);
   \node (1) at (.5,-.4)   {$\bf{1}$};
   \node (1) at (1.5,-.4)   {$\bf{2}$};
   \node (1) at (2.5,-.4)   {$\bf{3}$};
   \node (1) at (3.5,-.4)   {$\bf{4}$};
   \node (1) at (4.5,-.4)   {$\bf{5}$};
   \draw  [<-> ] (5.5,.5)-- (6.5,.5);
   \draw  (5,0) -- (5,1);
   \node (1) at (13,1.5)   {$\bf{c_3}$};
   \node (1) at (8,1.3)   {$\bf{-c_2}$};
   \node (1) at (8,.5)   {$\bf{1}$};
   \node (1) at (12,.7)   {$\bf{1}$};
   \node (1) at (14,.5)   {$\bf{1}$};
  \draw[->] (d)  to[bend left] (b);
  \draw[->] (b)  to[bend left] (c);
   \draw[->] (c)  to[bend left] (d);
   \draw[->] (a)  to[bend left] (e);
     \draw[->] (e)  to[bend left] (a);
  	\fill[black!100!](7.2,.53) circle (.05);
  	\fill[black!100!](8.87,.75) circle (.05);
  	\fill[black!100!](11.2,.78) circle (.05);
  	\fill[black!100!](13,.63) circle (.05);
  	\fill[black!100!](14.8,.5) circle (.05);
  	\end{tikzpicture}
  	\caption{ The bold line in each tile represents the length of the corresponding tile. } 
  	\label{fig:fibbo}
  \end{figure}
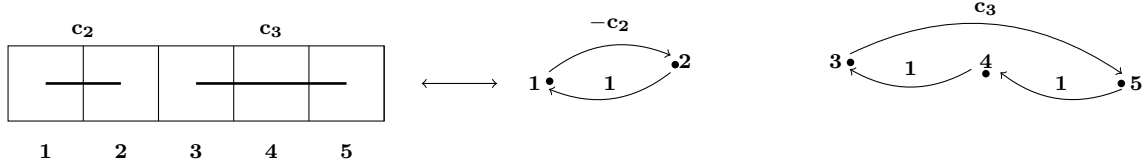
To complete the proof, we have to show that this bijection is sign and weight preserving.  So, suppose $L_{\tau}$ contains $k$ many even length cycles and $m$ many odd length cycles. Let the total number of vertices in $k$ cycles are $2s.$  The remaining $n-2s (n$ is total number of vertices in $L_{\tau}$ ) vertices are in $m$ odd cycles. Now, \[(-1)^{n-(k+m)}w(L_{\tau})=(-1)^k(-1)^{n-m}w(L_{\tau}).\] If $m$ is odd, then $n-2s$ is odd, so $n$ is odd. Again if $m$ is even, then $n$ is also even. So, $(-1)^{n-m}$ is always $1.$ Now observe that, only each even length cycle contributes negative sign to the weight of the LSD, $L_{\tau}$. Hence $(-1)^{n-(k+m)}w(L_{\tau})$ is positive. Again, from the construction of bijection, weight of $\tau$ is same as weight of the LSD, $L_{\tau}.$ For example, Figure \ref{fig:fibbo1}, describes the require bijection for finding $u_4,$ in the recurrence $u_n=c_1u_{n-1}+c_2u_{n-2}.$ 
 \end{proof}
   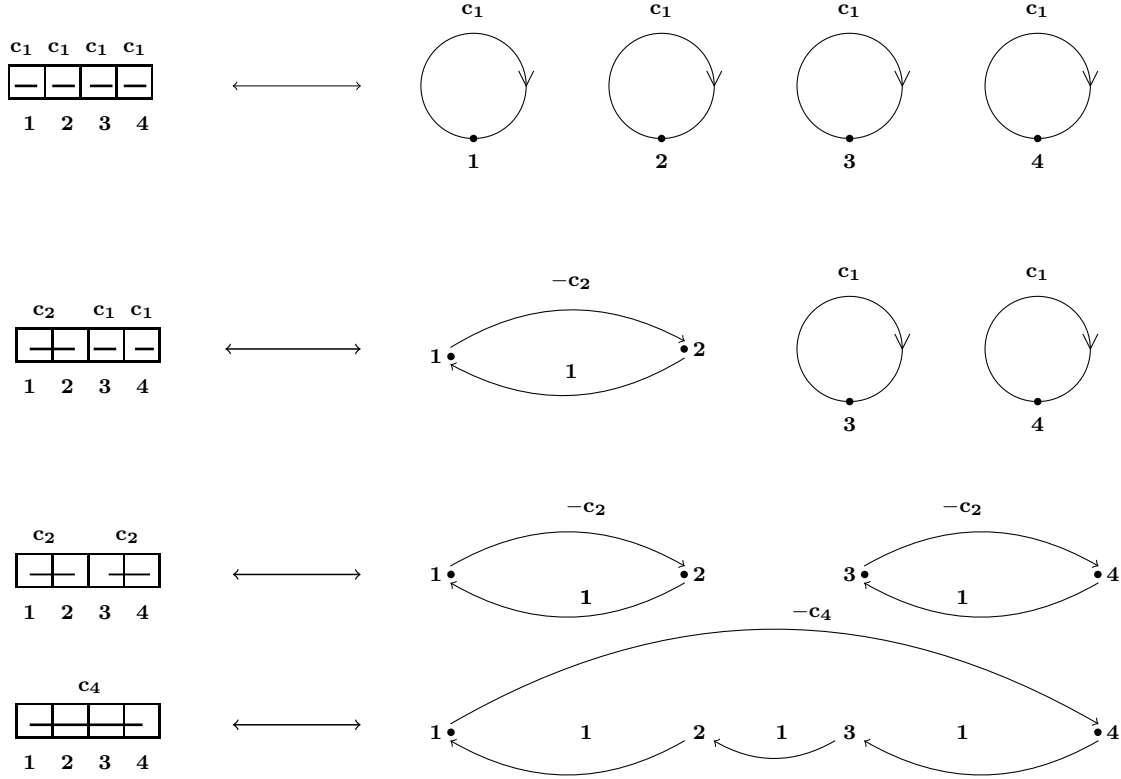
\begin{figure}[ht!]
   	\tiny
   	\tikzstyle{ver}=[]
   	\tikzstyle{vert}=[circle, draw, fill=black!100, inner sep=0pt, minimum width=4pt]
   	\tikzstyle{vertex}=[circle, draw, fill=black!.5, inner sep=0pt, minimum width=4pt]
   	\tikzstyle{edge} = [draw,thick,-]
   	\tikzstyle{node_style} = [circle,draw=blue,fill=blue!20!,font=\sffamily\Large\bfseries]
   \begin{tikzpicture}[scale=1]
   	\tikzstyle{edge_style} = [draw=black, line width=2mm, ]
   	\tikzstyle{node_style} = [draw=blue,fill=blue!00!,font=\sffamily\Large\bfseries]
   	\tikzset{
   		LabelStyle/.style = { rectangle, rounded corners, draw,
   			minimum width = 2em, fill = yellow!50,
   			text = red, font = \bfseries },
   		VertexStyle/.append style = { inner sep=5pt,
   			font = \Large\bfseries},
   		EdgeStyle/.append style = {->, bend left} }
   	\node (B1) at (-1.5,0)   {	$\begin{array}{cclll}
   		\begin{Young}
   		&&&\cr
   		\end{Young}&\hspace{.1cm}
   		\end{array}$};
   	\node (B1) at (-1.5,.5)   {$\bf{c_1}$};
   	\node (B1) at (-2,.5)   {$\bf{c_1}$};%
   	\node (B1) at (-2.4,-.5)   {$\bf{1}$};
   	\node (B1) at (-1.9,-.5)   {$\bf{2}$};
   	\node (B1) at (-1.4,-.5)   {$\bf{3}$};
   	\node (B1) at (-.9,-.5)   {$\bf{4}$};
   	\node (B1) at (-1,.5)   {$\bf{c_1}$};
   	\node (B1) at (-2.5,.5)   {$\bf{c_1}$};
   	\draw[line width=.35 mm ] (-1.6,0) -- (-1.3,0);
   \draw[line width=.35 mm  ] (-1.15,0) -- (-.85,0);
   	\draw[ line width=.35 mm ] (-2.1,0) -- (-1.8,0);
   	\draw[line width=.35 mm  ] (-2.6,0) -- (-2.3,0);
   	%
   	\draw[<->, line width=.1 mm] (0.3,0) -- (2,0);
   	\draw[](3.5,0) circle (.7);
   	\draw[ ] (4.2,0) -- (4.3,.2);
   	\draw[ ] (4.2,0) -- (4.1,.2);
   	\node (B1) at (3.5,-1)   {$\bf{1}$};
   	\fill[black!100!](3.5,-.7) circle (.05);
   	\draw[](6,0) circle (.7);
   	\draw[ ] (6.7,0) -- (6.8,.2);
   	\draw[ ] (6.7,0) -- (6.6,.2);
   	\node (B1) at (6,-1)   {$\bf{2}$};
   	\fill[black!100!](6,-.7) circle (.05);
   	\draw[](8.5,0) circle (.7);
   	\draw[ ] (9.2,0) -- (9.3,.2);
   	\draw[ ] (9.2,0) -- (9.1,.2);
   	\node (B1) at (8.5,-1)   {$\bf{3}$};
   	\fill[black!100!](8.5,-.7) circle (.05);
   	\draw[](11,0) circle (.7);
   	\draw[  ] (11.7,0) -- (11.8,.2);
   	\draw[  ] (11.7,0) -- (11.6,.2);
   	\node (B1) at (11,-1)   {$\bf{4}$};
   	\fill[black!100!](11,-.7) circle (.05);
   	\node (B1) at (3.5,1)   {$\bf{c_1}$};
   	\node (B1) at (6,1)   {$\bf{c_1}$};
   	\node (B1) at (8.5,1)   {$\bf{c_1}$};
   	\node (B1) at (11,1)   {$\bf{c_1}$};
   	\node (B1) at (-1.4,-3.5)   {$\begin{array}{cclll}
   		\begin{Young}
   		&&&\cr
   		\end{Young}&\hspace{.1cm}
   		\end{array}$};
   	\node (B1) at (-1.4,-3)   {$\bf{c_1}$};
   	\node (B1) at (-.9,-3)   {$\bf{c_1}$};
   	\node (B1) at (-2.2,-3)   {$\bf{c_2}$}; 
   	\draw[line width=.35 mm  ] (-1.55,-3.5) -- (-1.25,-3.5);
   	\draw[line width=.35 mm  ] (-1,-3.5) -- (-.75,-3.5);
   \node (B1) at (-2.4,-4)   {$\bf{1}$};
   \node (B1) at (-1.9,-4)   {$\bf{2}$};
   \node (B1) at (-1.4,-4)   {$\bf{3}$};
   \node (B1) at (-.9,-4)   {$\bf{4}$};
   	\draw[<->, line width=.2 mm] (0.2,-3.5) -- (2,-3.5);
   	\draw[](8.5,-3.5) circle (.7);
   	\draw[ ] (9.2,-3.5) -- (9.1,-3.3);
   	\draw[ ] (9.2,-3.5) -- (9.3,.-3.3);
   	\node (B1) at (8.5,-4.5)   {$\bf{3}$};
   	\fill[black!100!](8.5,-4.2) circle (.05);
   	\draw[line width=.35 mm](-2.4,-3.5)--(-1.8, -3.5);
   	\node (B1) at (8.5,-2.5)   {$\bf{c_1}$};
   	\draw[](11,-3.5) circle (.7);
   	\draw[ ] (11.7,-3.5) -- (11.6,-3.3);
   	\draw[ ] (11.7,-3.5) -- (11.8,-3.3);
   	\node (B1) at (11,-4.5)   {$\bf{4}$};
   	\fill[black!100!](11,-4.2) circle (.05);
   	\node (B1) at (11, -2.5)   {$\bf{c_1}$};
   	\node (B1) at (-1.4,-6.5)  {$\begin{array}{cclll}
   		\begin{Young}
   	&&&\cr
   	\end{Young}&\hspace{.1cm}
   		\end{array}$};
   	\node (B1) at (-2.2,-6)   {$\bf{c_2}$};
   	\node (B1) at (-1.1,-6)   {$\bf{c_2}$};
   	\node (B1) at (-2.4,-7)   {$\bf{1}$};
   	\node (B1) at (-1.9,-7)   {$\bf{2}$};
   	\node (B1) at (-1.4,-7)   {$\bf{3}$};
   	\node (B1) at (-.9,-7)   {$\bf{4}$};
   	%
   	\draw[<->, line width=.2 mm] (0.3,-6.5) -- (2,-6.5);
   \draw[line width=.25 mm](-2.4,-6.5)--(-1.8, -6.5);
   	\draw[line width=.25 mm](-1.35,-6.5)--(-0.8, -6.5);
   	\node[] (a) at  (3,-3.6) {$\bf{1}$};
   	\node[] (b) at  (6.5,-3.5) {$\bf{2}$};
   	\fill[black!100!](3.2,-3.6) circle (.05);
   	\fill[black!100!](6.3,-3.5) circle (.05);
   	\draw[->] (a)  to[bend left] (b);
   	\draw[->] (b)  to[bend left] (a);
   	\node (B1) at (4.8, -2.6)   {$\bf{-c_2}$};
   	\node (B1) at (4.8, -3.8)   {$\bf{1}$};
   	\node[] (c) at  (3,-6.5) {$\bf{1}$};
   	\node[] (d) at  (6.5,-6.5) {$\bf{2}$};
   	\node[] (e) at  (8.5,-6.5) {$\bf{3}$};
   	\node[] (f) at  (12,-6.5) {$\bf{4}$};
   	\fill[black!100!](3.2,-6.5) circle (.05);
   	\fill[black!100!](6.3,-6.5) circle (.05);
   	\fill[black!100!](8.7,-6.5) circle (.05);
   	\fill[black!100!](11.8,-6.5) circle (.05);
   	\draw[->] (c)  to[bend left] (d);
   	\draw[->] (d)  to[bend left] (c);
   	\draw[->] (e)  to[bend left] (f);
   	\draw[->] (f)  to[bend left] (e);
   	\node (B1) at (5, -5.6)   {$\bf{-c_2}$}; 
   	\node (B1) at (5, -6.8)   {$\bf{1}$}; 
   	\node (B1) at (10, -5.6)   {$\bf{-c_2}$};
   	\node (B1) at (10, -6.8)   {$\bf{1}$};
   	%
   	\node (B1) at (-1.4,-8.5)  {$\begin{array}{cclll}
   		\begin{Young}
   		&&&\cr
   		\end{Young}&\hspace{.1cm}
   		\end{array}$};
   	\node (B1) at (-1.6,-8)   {$\bf{c_4}$};
   	\node (B1) at (-2.4,-9)   {$\bf{1}$};
   	\node (B1) at (-1.9,-9)   {$\bf{2}$};
   	\node (B1) at (-1.4,-9)   {$\bf{3}$};
   	\node (B1) at (-.9,-9)   {$\bf{4}$};
   	\draw[<->, line width=.2 mm] (0.3,-8.5) -- (2,-8.5);
   	\draw[line width=.35 mm](-2.4,-8.5)--(-.9, -8.5);
   	\node[] (a) at  (3,-8.6) {$\bf{1}$};
   	\node[] (d) at  (12,-8.6) {$\bf{4}$};
   	\node[] (b) at  (6.5,-8.6) {$\bf{2}$};
   	\node[] (c) at  (8.5,-8.6) {$\bf{3}$};
   	\fill[black!100!](3.2,-8.6) circle (.05);
   	\fill[black!100!](11.8,-8.6) circle (.05);
   	\draw[->] (a)  to[bend left] (d);
   	\draw[->] (d)  to[bend left] (c);
   	\draw[->] (c)  to[bend left] (b);
   	\draw[->] (b)  to[bend left] (a);
   	 \node (B1) at (5, -8.6)   {$\bf{1}$}; 
   	\node (B1) at (5, -6.8)   {$\bf{1}$}; 
   	 \node (B1) at (8, -7)   {$\bf{-c_4}$};
   	\node (B1) at (10, -8.6)   {$\bf{1}$};
   	\node (B1) at (7.6, -8.6)   {$\bf{1}$};
\end{tikzpicture}
\caption{ Each bold line in every tile represents the length of the corresponding tile.  For example, the last tilling of the board contains exactly one tile of length $4$ and occupying the position $1, 2, 3, 4,$ whereas the second tilling contains one tile of length $2$ occupying the position $1, 2$ and two $1$ length tiles occupying the positions $3$ and $4.$ Numbers appearing above the tiles is the weights of the corresponding tiles. The numbers appearing above the edges are weights of the corresponding edges. }
   	\label{fig:fibbo1}
   \end{figure}
\begin{defn}\label{r-nacci num}
The $n$-th term $F_n,$ of the $r$-acci numbers  defined by the recurrence	\[F_n=F_{n-1}+F_{n-2}+\cdots+F_{n-r},\hspace{.5cm} n\geq 2, F_0=1, F_{i}=0, i<0.\] 
\end{defn} 
One combinatorial interpretation  of $r$-acci numbers $F_n$ is, the number of ways to tile an $n$-board using tiles of length at most $r.$ In \cite{26}, the authors proved a formula for the $r$-acci numbers. Here we give a determinantal expression of $r$-acci numbers. Consider an $n\times n$ matrix 
\begin{equation}\label{matrix G} 	
 	G= \left(\begin{array}{ccccccccc}
 	1 & -1 & 1 & \cdots & (-1)^{r+1} & 0 &   \cdots  & 0  & 0 \\
 	1& 1 & -1 & \cdots & (-1)^r & (-1)^{r+1}&  \cdots & 0  &0\\
 	0 & 1 & 1 & \cdots & (-1)^{r-1} & (-1)^r&\cdots&0&0 \\
 	\vdots & \vdots & \vdots &\ddots & \vdots & \vdots  &\ddots&\vdots&\vdots\\
 	0&0&0&\cdots&1&-1&\cdots& (-1)^{n-r+1}&(-1)^{n-r+2}\\
 	0&0&0&\cdots&1&1&\cdots& (-1)^{n-r}&(-1)^{n-r+1}\\
 	\vdots & \vdots & \vdots &\ddots & \vdots & \vdots  &\ddots&\vdots&\vdots\\
 	0 & 0 &0 & \cdots &0& 0&  \cdots & 1&-1 \\
 	0 & 0 & 0 & \cdots & 0&0  &\cdots& 1 & 1 \\
 	\end{array}
 	\right).
 	\end{equation}
 	\begin{corollary}\label{gen fibo}
 Let $F_n$ be the $n$-th term of  the $r$-acci numbers defined as \eqref{r-nacci num}. Then $F_n=\det(G).$
\end{corollary}
 	\begin{proof}
The proof of corollary follows from combinatorial interpretation the $r$-acci numbers (stated above) and the proof of Theorem \ref{gen lrr the}.  
 \end{proof}
 \begin{corollary}[Corollary $5$ \cite{26}]
 Let $F_n$ be the $n$-th term of the $r$-acci numbers. Then	\[F_n=\sum\limits_{2i_2+\cdots+ri_r \leq n}\frac{(n-i_2-2i_3-\cdots-(r-1)i_r)!}{i_2!\cdots i_r!((n-2i_2-3i_3-\cdots-ri_r)!)}.\]
\end{corollary}
 \begin{proof}
 By Corollary \ref{gen fibo}, $F_n=\det(G)$ and each LSD of $D(G)$ contributes $1$ to $\det(G).$ Again, \[\frac{(n-i_2-2i_3-\cdots-(r-1)i_r)!}{i_2!\cdots i_r!((n-2i_2-3i_3-\cdots-ri_r)!)}\] is the number of LSD containing $(n-i_2-2i_3-\cdots-(r-1)i_r)!$ many loops and $i_t (t=2, 3, \cdots, r)$ many cycles of length $t.$  Hence   \[\det(G)=\sum\limits_{2i_2+\cdots+ri_r \leq n}\frac{(n-i_2-2i_3-\cdots-(r-1)i_r)!}{i_2!\cdots i_r!((n-2i_2-3i_3-\cdots-ri_r)!)}.\]  
 \end{proof}
 The $n$-th term Fibonacci number $f_n$ satisfies the   recurrence
\begin{equation}\label{fibo recc}
 f_n = f_{n-1} + f_{n-2}, \hspace{.4cm} n\geq 2, f_0=1, f_1=1.
\end{equation}
A combinatorial interpreation of Fibonacci numbers $f_n$ is, the number of tillings of an $n$-board using tiles of length at most $2.$ Now consider an $n\times n$ matrix 
\begin{equation}\label{Fibbo matrix}
F=\left(
\begin{array}{cccccc}
1 & -1 & 0  &\cdots  & 0 & 0\\
1 & 1 & -1   &\cdots & 0 & 0 \\
0 & 1  & 1 & \cdots & 0  & 0  \\
\vdots & \vdots & \vdots & \ddots & \vdots & \vdots \\
0 & 0 & 0 & \cdots  & 1 & -1  \\
0 & 0 & 0 & \cdots & 1 & 1 \\
\end{array}
\right).
\end{equation}
 Then we have the following corollary about Fibonacci numbers.
 \begin{corollary}\label{fibocor}
 Let $F$ be the matrix defined as \eqref{Fibbo matrix}. Then $\det(F)$ gives the $n$-th Fibonacci number, $f_n.$
 \end{corollary}
 \begin{proof}
 We use the same argument as in the proof of Theorem \ref{gen lrr the} and the combinatorial interpretation of Fibonacci numbers as described above.    
\end{proof}
\begin{corollary}[Binet's Formula]
 	Let $f_n$ be the $n$-th Fibonacci numbers defined as \ref{fibo recc}. Then Binet's formula says that \[f_n=\frac{\left( \frac{1+\sqrt{5}}{2}\right)^{n+1}-\left( \frac{1-\sqrt{5}}{2}\right)^{n+1}}{\sqrt{5}}.\]
 \end{corollary}
 \begin{proof}
 	To prove the Binet's formula we consider the matrix described as \eqref{bienet} in Corollary \ref{cor3}. If $x\neq y,$ then the determinant of the matrix \eqref{bienet} can be written as $\frac{x^{n+1}-y^{n+1}}{x-y}.$ Now, we put $x+y=1$ and $xy=-1$ in the matrix \eqref{bienet}. Then we get the matrix  $F.$ Clearly $x, y$ are the roots of the polynomial $t^2-t-1=0.$ So, $x=\frac{1+\sqrt 5}{2}$ and $y=\frac{1-\sqrt 5}{2}.$  By Corollary \ref{fibocor},  $f_n=\det(F).$  Again by Corollary \ref{cor3}, $\det(F)=\frac{\left( \frac{1+\sqrt{5}}{2}\right)^{n+1}-\left( \frac{1-\sqrt{5}}{2}\right)^{n+1}}{\sqrt{5}}.$  \[\text{ Hence }f_n=\frac{\left( \frac{1+\sqrt{5}}{2}\right)^{n+1}-\left( \frac{1-\sqrt{5}}{2}\right)^{n+1}}{\sqrt{5}}.\]  
 \end{proof}
 \section{determinantal expression of lucas numbers and a new identity }
 In this section we derive a determinantal formula for  Lucas numbers. As a recipe to do so, we first prove a new determinantal identity and as a consequences we get the desired determinantal formula for Lucas numbers. Lucas numbers are defined by the following recurrence \[\ell_n=\ell_{n-1}+\ell_{n-2}, n\geq 3, \ell_0=2, \ell_1=1, \ell_2=3.\]
A well known combinatorial interpretation of Lucas numbers (see \cite{24}) $\ell_n,$ is the number of ways to tile a circular board composed of $n$ labeled cells with $1$-board and $2$-board . 

 For two variables  $a, b $ we consider an $n\times n (n\geq 3)$ matrix
 \begin{equation}\label{lucas gen matrix}
 S=	\left(\begin{array}{cccccc}
 a+b & (-1)^{n+1}a & 0  &\cdots  & 0 &b\\
 (-1)^{n+1}b & a+b & a   &\cdots & 0 & 0 \\
 0 & b  & a+b & \cdots & 0  & 0  \\
 \vdots & \vdots & \vdots & \ddots & \vdots & \vdots \\
 0 & 0 & 0 & \cdots  & a+b & a  \\
 a & 0 & 0 & \cdots & b & a+b \\
 \end{array}
 \right).
 \end{equation}
 Then the following holds.  
 \begin{theorem}\label{gen iden lucas}
Let $S$ be an $n\times n$ matrix defined as \eqref{lucas gen matrix}. Then $\det(S)=2(a^n+b^n).$	
 \end{theorem}
 \begin{proof}
We prove this theorem for the case, $n=4$ (we can prove the general case in a similar way). For this case,  \[S=\left(\begin{array}{cccc}
 	a+b & -a&0  & b\\
 	-b & a+b  & a&0 \\
 	0&b&a+b&a\\
 	a &0& b & a+b \\
 	\end{array}
 	\right).\] \text{ Now, we have to prove that, } $\det(S)=2(a^4+b^4).$ Think $A=\{ a, b\}$ to be the set of letters. Let $C$ be the set of all cyclic words of length $4$ formed by $a, b.$ For cyclic words, we always take the starting point to be $1$ and the orientation to be clockwise. For example, Figure \ref{exam cyclic words}, contains three cyclic words of length $4,$ with starting point $1$ and orientation clockwise.     
 	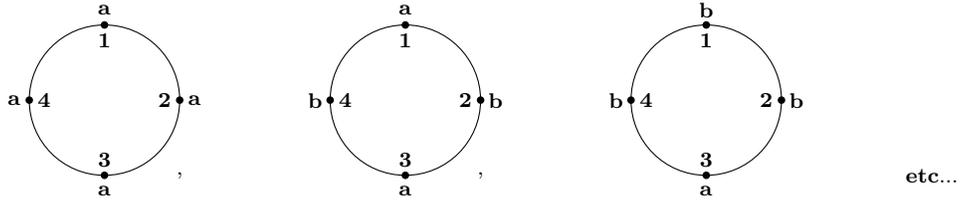
\begin{figure}[ht!] 
 		\tiny 
 		\tikzstyle{ver}=[]
 		\tikzstyle{vert}=[circle, draw, fill=black!100, inner sep=0pt, minimum width=4pt]
 		\tikzstyle{vertex}=[circle, draw, fill=black!.5, inner sep=0pt, minimum width=4pt]
 		\tikzstyle{edge} = [draw,thick,-]
 		\tikzstyle{node_style} = [circle,draw=blue,fill=blue!20!,font=\sffamily\Large\bfseries]
 		\centering
 		\tikzset{->,>=stealth', auto,node distance=1cm,
 			thick,main node/.style={circle,draw,font=\sffamily\Large\bfseries}}
 		\tikzset{->-/.style={decoration={
 					markings,
 					mark=at position #1 with {\arrow{>}}},postaction={decorate}}}
 		
 		\begin{tikzpicture}[scale=1]
 		\tikzstyle{edge_style} = [draw=black, line width=2mm, ]
 		\tikzstyle{node_style} = [draw=blue,fill=blue!00!,font=\sffamily\Large\bfseries]
 		\tikzset{
 			LabelStyle/.style = { rectangle, rounded corners, draw,
 				minimum width = 2em, fill = yellow!50,
 				text = red, font = \bfseries },
 			VertexStyle/.append style = { inner sep=5pt,
 				font = \Large\bfseries},
 			EdgeStyle/.append style = {->, bend left} }
 		\tikzset{vertex/.style = {shape=circle,draw,minimum size=1.5em}}
 		\tikzset{edge/.style = {->,> = latex'}}
 		\draw[](-3,0) circle (1);
 		\fill[black!100!] (-3,-1) circle (.05);
 		\fill[black!100!] (-3,1.) circle (.05);
 		\fill[black!100!] (-4,0) circle (.05);
 		\fill[black!100!] (-2,0) circle (.05);
 		\node (B1) at (-3,-1.2)   {$\bf{a}$};
 		\node (B1) at (-3,1.2)   {$\bf{a}$};	
 		\node (B1) at (-1.8,0)   {$\bf{a}$};
 		\node (B1) at (-4.2,0)   {$\bf{a}$};
 		\node (B1) at (-3,.8)   {$\bf{1}$};
 		\node (B1) at (-2.2,0)   {$\bf{2}$};	
 		\node (B1) at (-3,-.8)   {$\bf{3}$};
 		\node (B1) at (-3.8, 0)   {$\bf{4}$};
 		\draw[](1,0) circle (1);
 		\fill[black!100!] (1,1) circle (.05);
 		\fill[black!100!] (2,0) circle (.05);
 		\fill[black!100!] (1,-1) circle (.05);
 		\fill[black!100!] (0,0) circle (.05);
 		\node (B1) at (1,1.2)   {$\bf{a}$};
 		\node (B1) at (2.2, 0)   {$\bf{b}$};	
 		\node (B1) at (1,-1.2)   {$\bf{a}$};
 		\node (B1) at (-.2,0)   {$\bf{b}$};
 		\node (B1) at (1,.8)   {$\bf{1}$};
 		\node (B1) at (1.8,0)   {$\bf{2}$};	
 		\node (B1) at (1,-.8)   {$\bf{3}$};
 		\node (B1) at (.2, 0)   {$\bf{4}$};
 		\node (B1) at (2, -1)   {$\bf{,}$};
 		\node (B1) at (-2, -1)   {$\bf{,}$};
 		\draw[](5,0) circle (1);
 		\fill[black!100!] (5,1) circle (.05);
 		\fill[black!100!] (6,0) circle (.05);
 		\fill[black!100!] (5,-1) circle (.05);
 		\fill[black!100!] (4,0) circle (.05);
 		\node (B1) at (5,1.2)   {$\bf{b}$};
 		\node (B1) at (6.2, 0)   {$\bf{b}$};	
 		\node (B1) at (5,-1.2)   {$\bf{a}$};
 		\node (B1) at (3.8,0)   {$\bf{b}$};
 		\node (B1) at (5,.8)   {$\bf{1}$};
 		\node (B1) at (5.8,0)   {$\bf{2}$};	
 		\node (B1) at (5,-.8)   {$\bf{3}$};
 		\node (B1) at (4.2, 0)   {$\bf{4}$};
 		\node (B1) at (8, -1)   {$\bf{etc...}$};
 		\end{tikzpicture}
 		\caption{This figure contains three cyclic words  formed by $a, b.$}
 		\label{exam cyclic words}
 	\end{figure}  
Now, let us introduce some notations. We will denote by $\mathbb{Z}C(n)$ the set of all formal linear combinations of cyclic words of length $n$ (with the starting point and orientation, already prescribed) with integer coefficients. For example, $c_1,c_2, c_1+c_2, 2c_1-c_2, 0 (\text{ the empty word })$ are some typical elements of $\mathbb{Z}C(n),$ where $c_1, c_2$ are cyclic words. Now let us consider Figure \ref{sum cyclic words}. One  can notice that none of the items in the left hand side of Figure \ref{sum cyclic words} is a cyclic word but can be thought of as an element of $\mathbb{Z}C(n)$ by the right hand side of the corresponding item. 
 \begin{figure}[ht!] 
 	\tiny 
 	\tikzstyle{ver}=[]
 	\tikzstyle{vert}=[circle, draw, fill=black!100, inner sep=0pt, minimum width=4pt]
 	\tikzstyle{vertex}=[circle, draw, fill=black!.5, inner sep=0pt, minimum width=4pt]
 	\tikzstyle{edge} = [draw,thick,-]
 	\tikzstyle{node_style} = [circle,draw=blue,fill=blue!20!,font=\sffamily\Large\bfseries]
 	\centering
 	\tikzset{->,>=stealth', auto,node distance=1cm,
 		thick,main node/.style={circle,draw,font=\sffamily\Large\bfseries}}
 	\tikzset{->-/.style={decoration={
 				markings,
 				mark=at position #1 with {\arrow{>}}},postaction={decorate}}}
 	\begin{tikzpicture}[scale=1]
 	\tikzstyle{edge_style} = [draw=black, line width=2mm, ]
 	\tikzstyle{node_style} = [draw=blue,fill=blue!00!,font=\sffamily\Large\bfseries]
 	\tikzset{
 		LabelStyle/.style = { rectangle, rounded corners, draw,
 			minimum width = 2em, fill = yellow!50,
 			text = red, font = \bfseries },
 		VertexStyle/.append style = { inner sep=5pt,
 			font = \Large\bfseries},
 		EdgeStyle/.append style = {->, bend left} }
 	\tikzset{vertex/.style = {shape=circle,draw,minimum size=1.5em}}
 	\tikzset{edge/.style = {->,> = latex'}}
 	\draw[](-3,3.5) circle (1);
 	\fill[black!100!] (-3,4.5) circle (.05);
 	\fill[black!100!] (-2,3.5) circle (.05);
 	\fill[black!100!] (-3,2.5) circle (.05);
 	\fill[black!100!] (-4,3.5) circle (.05);
 	\node (B1) at (-3,4.7)   {$\bf{a+b}$};
 	\node (B1) at (-1.8,3.5)   {$\bf{b}$};	
 	\node (B1) at (-3,2.3)   {$\bf{a}$};
 	\node (B1) at (-4.2,3.5)   {$\bf{b}$};
 	\node (B1) at (-3,4.3)   {$\bf{1}$};
 	\node (B1) at (-2.2,3.5)   {$\bf{2}$};	
 	\node (B1) at (-3,2.7)   {$\bf{3}$};
 	\node (B1) at (-3.8,3.5)   {$\bf{4}$};
 	%
 	\node(B1) at (-.9, 3.5) {$\bf{=}$};
 	\draw[](.7,3.5) circle (1);
 	\fill[black!100!] (.7,4.5) circle (.05);
 	\fill[black!100!] (1.7,3.5) circle (.05);
 	\fill[black!100!] (.7,2.5) circle (.05);
 	\fill[black!100!] (-.3,3.5) circle (.05);
 	\node (B1) at (.7,4.7)   {$\bf{a}$};
 	\node (B1) at (1.9, 3.5)   {$\bf{b}$};	
 	\node (B1) at (.7,2.3)   {$\bf{a}$};
 	\node (B1) at (-.5, 3.5)   {$\bf{b}$};
 	\node (B1) at (.7,4.3)   {$\bf{1}$};
 	\node (B1) at (1.5, 3.5)   {$\bf{2}$};	
 	\node (B1) at (.7,2.7)   {$\bf{3}$};
 	\node (B1) at (-.1,3.5)   {$\bf{4}$};
 	\node(B1) at (2.35, 3.5) {$\bf{+}$};
 	%
 	%
 	%
 	\draw[](4,3.5) circle (1);
 	\fill[black!100!] (4,4.5) circle (.05);
 	\fill[black!100!] (5,3.5) circle (.05);
 	\fill[black!100!] (4,2.5) circle (.05);
 	\fill[black!100!] (3,3.5) circle (.05);
 	\node (B1) at (4,4.7)   {$\bf{b}$};
 	\node (B1) at (5.2, 3.5)   {$\bf{b}$};	
 	\node (B1) at (4,2.3)   {$\bf{a}$};
 	\node (B1) at (2.8, 3.5)   {$\bf{b}$};
 	\node (B1) at (4,4.3)   {$\bf{1}$};
 	\node (B1) at (4.8, 3.5)   {$\bf{2}$};	
 	\node (B1) at (4,2.7)   {$\bf{3}$};
 	\node (B1) at (3.2,3.5)   {$\bf{4}$};
 	\draw[](-3,0) circle (1);
 	\fill[black!100!] (-3,-1) circle (.05);
 	\fill[black!100!] (-3,1) circle (.05);
 	\fill[black!100!] (-2,0) circle (.05);
 	\fill[black!100!] (-4,0) circle (.05);
 	\node (B1) at (-3,-1.2)   {$\bf{a}$};
 	\node (B1) at (-3,1.2)   {$\bf{a+b}$};	
 	\node (B1) at (-4.2,0)   {$\bf{a}$};
 	\node (B1) at (-1.5,0)   {$\bf{a+b}$};
 	\node (B1) at (-3,-.8)   {$\bf{3}$};
 	\node (B1) at (-3,.8)   {$\bf{1}$};	
 	\node (B1) at (-3.8,0)   {$\bf{4}$};
 	\node (B1) at (-2.2,0)   {$\bf{2}$};
 	\node(B1) at (-.9, 0) {$\bf{=}$};
 	\draw[](.7,0) circle (1);
 	\fill[black!100!] (.7,1) circle (.05);
 	\fill[black!100!] (1.7,0) circle (.05);
 	\fill[black!100!] (.7,-1) circle (.05);
 	\fill[black!100!] (-.3,0) circle (.05);
 	\node (B1) at (.7,1.2)   {$\bf{a}$};
 	\node (B1) at (1.9, 0)   {$\bf{a}$};	
 	\node (B1) at (.7,-1.2)   {$\bf{a}$};
 	\node (B1) at (-.5, 0)   {$\bf{a}$};
 	\node (B1) at (.7,.8)   {$\bf{1}$};
 	\node (B1) at (1.5, 0)   {$\bf{2}$};	
 	\node (B1) at (.7,-.8)   {$\bf{3}$};
 	\node (B1) at (-.1, 0)   {$\bf{4}$};
 	\node(B1) at (2.35, 0) {$\bf{+}$};
 	\draw[](4,0) circle (1);
 	\fill[black!100!] (4,1) circle (.05);
 	\fill[black!100!] (5,0) circle (.05);
 	\fill[black!100!] (4,-1) circle (.05);
 	\fill[black!100!] (3,0) circle (.05);
 	\node (B1) at (4,1.2)   {$\bf{a}$};
 	\node (B1) at (5.2, 0)   {$\bf{b}$};	
 	\node (B1) at (4,-1.2)   {$\bf{a}$};
 	\node (B1) at (2.8, 0)   {$\bf{a}$};
 	\node (B1) at (4,.8)   {$\bf{1}$};
 	\node (B1) at (4.8, 0)   {$\bf{2}$};	
 	\node (B1) at (4,-.8)   {$\bf{3}$};
 	\node (B1) at (3.2, 0)   {$\bf{4}$};
 	%
 	\node(B1) at (5.8, 0) {$\bf{+}$};
 	\draw[](7.5,0) circle (1);
 	\fill[black!100!] (7.5,1) circle (.05);
 	\fill[black!100!] (8.5,0) circle (.05);
 	\fill[black!100!] (7.5,-1) circle (.05);
 	\fill[black!100!] (6.5,0) circle (.05);
 	\node (B1) at (7.5,1.2)   {$\bf{b}$};
 	\node (B1) at (8.7, 0)   {$\bf{a}$};	
 	\node (B1) at (7.5,-1.2)   {$\bf{a}$};
 	\node (B1) at (6.3, 0)   {$\bf{a}$};
 	\node (B1) at (7.5,.8)   {$\bf{1}$};
 	\node (B1) at (8.3, 0)   {$\bf{2}$};	
 	\node (B1) at (7.5,-.8)   {$\bf{3}$};
 	\node (B1) at (6.7, 0)   {$\bf{4}$};
 	\node(B1) at (9.2, 0) {$\bf{+}$};
 	%
 	%
 	%
 	\draw[](11,0) circle (1);
 	\fill[black!100!] (11,1) circle (.05);
 	\fill[black!100!] (12,0) circle (.05);
 	\fill[black!100!] (11,-1) circle (.05);
 	\fill[black!100!] (10,0) circle (.05);
 	\node (B1) at (11,1.2)   {$\bf{b}$};
 	\node (B1) at (12.2, 0)   {$\bf{b}$};	
 	\node (B1) at (11,-1.2)   {$\bf{a}$};
 	\node (B1) at (9.8, 0)   {$\bf{a}$};
 	\node (B1) at (11,.8)   {$\bf{1}$};
 	\node (B1) at (11.8, 0)   {$\bf{2}$};	
 	\node (B1) at (11,-.8)   {$\bf{3}$};
 	\node (B1) at (10.2, 0)   {$\bf{4}$};
 	%
 	%
 	%
 	\draw[](-3,-3) circle (1);
 	\fill[black!100!] (-3,-2) circle (.05);
 	\fill[black!100!] (-2,-3) circle (.05);
 	\fill[black!100!] (-3,-4) circle (.05);
 	\fill[black!100!] (-4,-3) circle (.05);
 	\node (B1) at (-3,-1.8)   {$\bf{a+a}$};
 	\node (B1) at (-1.8,-3)   {$\bf{b}$};	
 	\node (B1) at (-3,-4.2)   {$\bf{a}$};
 	\node (B1) at (-4.2, -3)   {$\bf{a}$};
 	\node (B1) at (-3,-2.2)   {$\bf{1}$};
 	\node (B1) at (-2.2,-3)   {$\bf{2}$};	
 	\node (B1) at (-3,-3.8)   {$\bf{3}$};
 	\node (B1) at (-3.8,-3)   {$\bf{4}$};
 	\node(B1) at (-.9, -3) {$\bf{=}$};
 	\draw[](.7,-3) circle (1);
 	\fill[black!100!] (.7,-2) circle (.05);
 	\fill[black!100!] (1.7,-3) circle (.05);
 	\fill[black!100!] (.7,-4) circle (.05);
 	\fill[black!100!] (-.3,-3) circle (.05);
 	\node (B1) at (.7,-1.8)   {$\bf{a}$};
 	\node (B1) at (1.9, -3)   {$\bf{b}$};	
 	\node (B1) at (.7,-4.2)   {$\bf{a}$};
 	\node (B1) at (-.5, -3)   {$\bf{a}$};
 	\node (B1) at (.7,-2.2)   {$\bf{1}$};
 	\node (B1) at (1.5, -3)   {$\bf{2}$};	
 	\node (B1) at (.7,-3.8)   {$\bf{3}$};
 	\node (B1) at (-.1, -3)   {$\bf{4}$};
 	\node(B1) at (2.35, -3) {$\bf{+}$};
 	%
 	%
 	%
 	\draw[](4,-3) circle (1);
 	\fill[black!100!] (4,-2) circle (.05);
 	\fill[black!100!] (5,-3) circle (.05);
 	\fill[black!100!] (4,-4) circle (.05);
 	\fill[black!100!] (3,-3) circle (.05);
 	\node (B1) at (4,-1.8)   {$\bf{a}$};
 	\node (B1) at (5.2, -3)   {$\bf{b}$};	
 	\node (B1) at (4,-4.2)   {$\bf{a}$};
 	\node (B1) at (2.8, -3)   {$\bf{a}$};
 	\node (B1) at (4,-2.2)   {$\bf{1}$};
 	\node (B1) at (4.8, -3)   {$\bf{2}$};	
 	\node (B1) at (4,-3.8)   {$\bf{3}$};
 	\node (B1) at (3.2, -3)   {$\bf{4}$};
 	%
 	%
 	%
 	\node (B1) at (12, -4)   {$\bf{etc...}$};
 	\end{tikzpicture}
 	\caption{}
 	\label{sum cyclic words}
 \end{figure}	
We want to evaluate the sum of all possible cyclic words of length $4,$ formed by $a, b,$ such that $ab$ does not appear as a sub word, i.e. $a$ and $b$ do not occur as a consecutive pair (note that $ba$ may appear as a sub word).
Figure \ref{avoiding word} shows some cyclic words containing $ab$ as a sub word.
 	\begin{figure}[ht!] 
 		\tiny 
 		\tikzstyle{ver}=[]
 		\tikzstyle{vert}=[circle, draw, fill=black!100, inner sep=0pt, minimum width=4pt]
 		\tikzstyle{vertex}=[circle, draw, fill=black!.5, inner sep=0pt, minimum width=4pt]
 		\tikzstyle{edge} = [draw,thick,-]
 		\tikzstyle{node_style} = [circle,draw=blue,fill=blue!20!,font=\sffamily\Large\bfseries]
 		\centering
 		\tikzset{->,>=stealth', auto,node distance=1cm,
 			thick,main node/.style={circle,draw,font=\sffamily\Large\bfseries}}
 		\tikzset{->-/.style={decoration={
 					markings,
 					mark=at position #1 with {\arrow{>}}},postaction={decorate}}}
 	\begin{tikzpicture}[scale=1]
 		\tikzstyle{edge_style} = [draw=black, line width=2mm, ]
 		\tikzstyle{node_style} = [draw=blue,fill=blue!00!,font=\sffamily\Large\bfseries]
 		\tikzset{
 			LabelStyle/.style = { rectangle, rounded corners, draw,
 				minimum width = 2em, fill = yellow!50,
 				text = red, font = \bfseries },
 			VertexStyle/.append style = { inner sep=5pt,
 				font = \Large\bfseries},
 			EdgeStyle/.append style = {->, bend left} }
 		\tikzset{vertex/.style = {shape=circle,draw,minimum size=1.5em}}
 		\tikzset{edge/.style = {->,> = latex'}}
 		\draw[](-3,0) circle (1);
 		\fill[black!100!] (-3,-1) circle (.05);
 		\fill[black!100!] (-3,1) circle (.05);
 		\fill[black!100!] (-2,0) circle (.05);
 		\fill[black!100!] (-4,0) circle (.05);
 		\node (B1) at (-3,-1.2)   {$\bf{b}$};
 		\node (B1) at (-3,1.2)   {$\bf{a}$};	
 		\node (B1) at (-4.2,0)   {$\bf{b}$};
 		\node (B1) at (-1.8,0)   {$\bf{b}$};
 		\node (B1) at (-3,-.8)   {$\bf{3}$};
 		\node (B1) at (-3,.8)   {$\bf{1}$};	
 		\node (B1) at (-3.8,0)   {$\bf{4}$};
 		\node (B1) at (-2.2,0)   {$\bf{2}$};
 		\node(B1) at (-1.8, -.8) {$\bf{,}$};
 		\draw[](2,0) circle (1);
 		\fill[black!100!] (2,1) circle (.05);
 		\fill[black!100!] (3,0) circle (.05);
 		\fill[black!100!] (2,-1) circle (.05);
 		\fill[black!100!] (1,0) circle (.05);
 		\node (B1) at (2,1.2)   {$\bf{a}$};
 		\node (B1) at (3.2, 0)   {$\bf{a}$};	
 		\node (B1) at (2,-1.2)   {$\bf{b}$};
 		\node (B1) at (.8, 0)   {$\bf{b}$};
 		\node (B1) at (2,.8)   {$\bf{1}$};
 		\node (B1) at (2.8, 0)   {$\bf{2}$};	
 		\node (B1) at (2,-.8)   {$\bf{3}$};
 		\node (B1) at (1.2, 0)   {$\bf{4}$};
 		\node(B1) at (3.2, -.8) {$\bf{,}$};
 		%
 		\draw[](7,0) circle (1);
 		\fill[black!100!] (7,1) circle (.05);
 		\fill[black!100!] (8,0) circle (.05);
 		\fill[black!100!] (7,-1) circle (.05);
 		\fill[black!100!] (6,0) circle (.05);
 		\node (B1) at (7,1.2)   {$\bf{b}$};
 		\node (B1) at (8.2, 0)   {$\bf{a}$};	
 		\node (B1) at (7,-1.2)   {$\bf{b}$};
 		\node (B1) at (5.8, 0)   {$\bf{a}$};
 		\node (B1) at (7,.8)   {$\bf{1}$};
 		\node (B1) at (7.8, 0)   {$\bf{2}$};	
 		\node (B1) at (7,-.8)   {$\bf{3}$};
 		\node (B1) at (6.2, 0)   {$\bf{4}$};
 		%
 		\node(B1) at (9.7, -1) {$\bf{etc...}$};
 		\end{tikzpicture}
 		\caption{These are some  cyclic words, where  $ab$  appears as a sub word. For example, in the first cyclic word the letters $a, b$ occupy two consecutive positions $1, 2$ respectively, whereas the last cyclic word contains two pairs of consecutive positions $2, 3$ and $4, 1$ occupied by $ab.$ }
 		\label{avoiding word}
\end{figure}
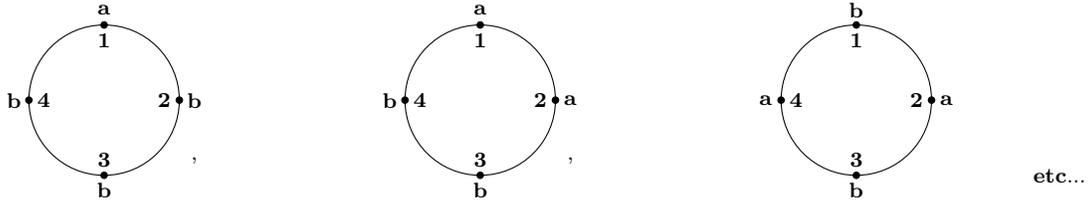
Let $C_0 (\in \mathbb{Z}C(4))$ be the formal sum of all cyclic words with constant coefficient $1,$ i.e., $\sum\limits_{\ell \in C}\ell.$ 
Now the sum of all possible cyclic words of length $4,$ such that the letters $a, b$ do not occupy two consecutive positions on the circle is clearly described by Figure \ref{req word}. Again, we can calculate this sum by PIE rule. By PIE rule, the  sum of all cyclic words avoiding $ab$ as a sub word is $C_0-C_1+C_2,$ where $C_1 $ is the sum of all cyclic words of length $4,$ formed by letters $a, b$ such that, there is at least one pair of consecutive positions (i.e. either $1, 2$ or $2, 3$ or $3, 4$ or $4, 1$ on circle,) occupied by $ab.$ For an illustration, see Figure \ref{avoiding word}. $C_2$ is the sum of all cyclic words of length $4,$ formed by letters $a, b$ such that, there is at least two pairs of consecutive positions (for example, $1, 2$ and $3, 4$ or $2, 3$ and $4, 1$ etc.,) occupied by $ab.$ See the last item of Figure  \ref{avoiding word} for an illustration. 
 \begin{figure}[ht!] 
\tiny  
\tikzstyle{ver}=[]
\tikzstyle{vert}=[circle, draw, fill=black!100, inner sep=0pt, minimum width=4pt]
 		\tikzstyle{vertex}=[circle, draw, fill=black!.5, inner sep=0pt, minimum width=4pt]
 		\tikzstyle{edge} = [draw,thick,-]
 		\tikzstyle{node_style} = [circle,draw=blue,fill=blue!20!,font=\sffamily\Large\bfseries]
 		\centering
 		\tikzset{->,>=stealth', auto,node distance=1cm,
 			thick,main node/.style={circle,draw,font=\sffamily\Large\bfseries}}
 		\tikzset{->-/.style={decoration={
 					markings,
 					mark=at position #1 with {\arrow{>}}},postaction={decorate}}}
 	\begin{tikzpicture}[scale=1]
 		\tikzstyle{edge_style} = [draw=black, line width=2mm, ]
 		\tikzstyle{node_style} = [draw=blue,fill=blue!00!,font=\sffamily\Large\bfseries]
 		\tikzset{
 			LabelStyle/.style = { rectangle, rounded corners, draw,
 				minimum width = 2em, fill = yellow!50,
 				text = red, font = \bfseries },
 			VertexStyle/.append style = { inner sep=5pt,
 				font = \Large\bfseries},
 			EdgeStyle/.append style = {->, bend left} }
 		\tikzset{vertex/.style = {shape=circle,draw,minimum size=1.5em}}
 		\tikzset{edge/.style = {->,> = latex'}}
 		\draw[](2,0) circle (1);
 		\fill[black!100!] (2,1) circle (.05);
 		\fill[black!100!] (3,0) circle (.05);
 		\fill[black!100!] (2,-1) circle (.05);
 		\fill[black!100!] (1,0) circle (.05);
 		\node (B1) at (2,1.2)   {$\bf{a}$};
 		\node (B1) at (3.2, 0)   {$\bf{a}$};	
 		\node (B1) at (2,-1.2)   {$\bf{a}$};
 		\node (B1) at (.8, 0)   {$\bf{a}$};
 		\node (B1) at (2,.8)   {$\bf{1}$};
 		\node (B1) at (2.8, 0)   {$\bf{2}$};	
 		\node (B1) at (2,-.8)   {$\bf{3}$};
 		\node (B1) at (1.2, 0)   {$\bf{4}$};
 		%
 		\draw[](6,0) circle (1);
 		\fill[black!100!] (6,1) circle (.05);
 		\fill[black!100!] (7,0) circle (.05);
 		\fill[black!100!] (6,-1) circle (.05);
 		\fill[black!100!] (5,0) circle (.05);
 		\node (B1) at (6,1.2)   {$\bf{b}$};
 		\node (B1) at (7.2, 0)   {$\bf{b}$};	
 		\node (B1) at (6,-1.2)   {$\bf{b}$};
 		\node (B1) at (4.8, 0)   {$\bf{b}$};
 		\node (B1) at (6,.8)   {$\bf{1}$};
 		\node (B1) at (6.8, 0)   {$\bf{2}$};	
 		\node (B1) at (6,-.8)   {$\bf{3}$};
 		\node (B1) at (5.2, 0)   {$\bf{4}$};
 		\node(B1) at (4, 0) {$\bf{+}$};
 		%
 		\end{tikzpicture}
 		\caption{}
 		\label{req word}
 	\end{figure}
Now we show a sign and weight preserving bijection between the terms in $C_0-C_1+C_2$ and LSD in the set $\mathcal{L}\setminus \{L_1, L_2\},$ where  $\mathcal{L}$ is the collection of all LSD in $D(S)$ and $L_1, L_2$ are two LSD described in Figure \ref{fig:f2}.
\begin{figure}[ht!]
	\tiny
	\tikzstyle{ver}=[]
	\tikzstyle{vert}=[circle, draw, fill=black!100, inner sep=0pt, minimum width=4pt]
	\tikzstyle{vertex}=[circle, draw, fill=black!.5, inner sep=0pt, minimum width=4pt]
	\tikzstyle{edge} = [draw,thick,-]
	\tikzstyle{node_style} = [circle,draw=blue,fill=blue!20!,font=\sffamily\Large\bfseries]
	\centering
	\tikzset{->,>=stealth', auto,node distance=1cm,
		thick,main node/.style={circle,draw,font=\sffamily\Large\bfseries}}
	\tikzset{->-/.style={decoration={
				markings,
				mark=at position #1 with {\arrow{>}}},postaction={decorate}}}
	\begin{tikzpicture}[scale=1]
	\tikzstyle{edge_style} = [draw=black, line width=2mm, ]
	\tikzstyle{node_style} = [draw=blue,fill=blue!00!,font=\sffamily\Large\bfseries]
	\tikzset{
		LabelStyle/.style = { rectangle, rounded corners, draw,
			minimum width = 2em, fill = yellow!50,
			text = red, font = \bfseries },
		VertexStyle/.append style = { inner sep=5pt,
			font = \Large\bfseries},
		EdgeStyle/.append style = {->, bend left} }
	\tikzset{vertex/.style = {shape=circle,draw,minimum size=1.5em}}
	\tikzset{edge/.style = {->,> = latex'}}
	\node[] (a) at  (3,-3) {$\bf{1}$};
	\node[] (b) at  (6.5,-3) {$\bf{2}$};
	\node[] (c) at  (9.5,-3) {$\bf{3}$};
	\node[] (d) at  (12.5,-3) {$\bf{4}$};
	\fill[black!100!](3.2,-3) circle (.05);
	\fill[black!100!](6.5,-2.8) circle (.05);
	\fill[black!100!](9.5,-2.8) circle (.05);
	\fill[black!100!](12.3,-3) circle (.05);
	\draw[edge] (a)  to[bend left] (b);
	\draw[edge] (b)  to[bend left] (c);
	\draw[edge] (c)  to[bend left] (d);
	\draw[edge] (d)  to[bend left] (a);
	\node (B1) at (4.8, -2.2)   {$\bf{-a}$};
	\node (B1) at (7.2, -2.4)   {$\bf{a}$};
	\node (B1) at (11.2, -2.3)   {$\bf{a}$};
	\node (B1) at (4.8, -3.5)   {$\bf{a}$};
	\node[] (c) at  (3,-6.9) {$\bf{1}$};
	\node[] (d) at  (6.5,-7.1) {$\bf{2}$};
	\node[] (e) at  (9.5,-7.1) {$\bf{3}$};
	\node[] (f) at  (12.5,-7.1) {$\bf{4}$};
	\fill[black!100!](3.2,-6.9) circle (.05);
	\fill[black!100!](6.5,-7.3) circle (.05);
	\fill[black!100!](9.5,-7.3) circle (.05);
	\fill[black!100!](12.3,-7.1) circle (.05);
	\draw[edge] (c)  to[bend left] (f);
	\draw[edge] (f)  to[bend left] (e);
	\draw[edge] (e)  to[bend left] (d);
	\draw[edge] (d)  to[bend left] (c);
	\node (B1) at (6, -5.5)   {$\bf{b}$}; 
	\node (B1) at (7.5, -7)   {$\bf{b}$}; 
	\node (B1) at (5.2, -7)   {$\bf{-b}$}; 
	\node (B1) at (11, -7.2)   {$\bf{b}$}; 
	\end{tikzpicture}
	\caption{This figure contains two LSD of $D(S)$ $L_1 $ (above) and $L_2$ (below).  The numbers appearing  on each edge  in the above figure is the weight of the corresponding edge.}
	\label{fig:f2}	
\end{figure}
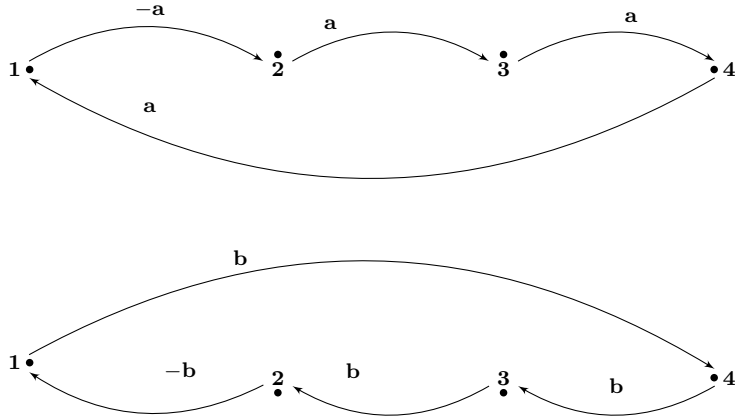
The \emph{weight of a cyclic word} $C,$ denoted by $w(C)$ is the product (here product is usual polynomial product) of all letters present in that cyclic word and extend this to $\mathbb{Z}C(n)$ by linearity, i.e., for any two cyclic words $\ell$ and $\acute{\ell},$ $w(x\ell+y\acute{\ell}):=xw(\ell)+yw(\acute{\ell}), x, y \in \mathbb{Z}.$ See Figures \ref{weight of word} and \ref{additivity of weight} for illustration.
\begin{figure}[ht!] 
 		\tiny 
 		\tikzstyle{ver}=[]
 		\tikzstyle{vert}=[circle, draw, fill=black!100, inner sep=0pt, minimum width=4pt]
 		\tikzstyle{vertex}=[circle, draw, fill=black!.5, inner sep=0pt, minimum width=4pt]
 		\tikzstyle{edge} = [draw,thick,-]
 		\tikzstyle{node_style} = [circle,draw=blue,fill=blue!20!,font=\sffamily\Large\bfseries]
 		\tikzset{round left paren/.style={ncbar=0.5cm,out=120,in=-120}}
 		\tikzset{round right paren/.style={ncbar=0.5cm,out=60,in=-60}}
 		\begin{tikzpicture}[scale=1]
 		\tikzstyle{edge_style} = [draw=black, line width=2mm, ]
 		\tikzstyle{node_style} = [draw=blue,fill=blue!00!,font=\sffamily\Large\bfseries]
 		\tikzset{
 			LabelStyle/.style = { rectangle, rounded corners, draw,
 				minimum width = 2em, fill = yellow!50,
 				text = red, font = \bfseries },
 			VertexStyle/.append style = { inner sep=5pt,
 				font = \Large\bfseries},
 			EdgeStyle/.append style = {->, bend left} }
 		\tikzset{vertex/.style = {shape=circle,draw,minimum size=1.5em}}
 		\tikzset{edge/.style = {->,> = latex'}}
 		\tikzset{
 			ncbar angle/.initial=90,
 			ncbar/.style={
 				to path=(\tikztostart)
 				-- ($(\tikztostart)!#1!\pgfkeysvalueof{/tikz/ncbar angle}:(\tikztotarget)$)
 				-- ($(\tikztotarget)!($(\tikztostart)!#1!\pgfkeysvalueof{/tikz/ncbar angle}:(\tikztotarget)$)!\pgfkeysvalueof{/tikz/ncbar angle}:(\tikztostart)$)
 				-- (\tikztotarget)
 			},
 			ncbar/.default=0.5cm,
 		}
 		\draw[](-3,0) circle (1);
 		\node (B1) at (-4.8,0)   {$\bf{w}$};
 		\draw [] (-4.2,-1.2) to [round left paren ] (-4.2,1.2);
 		\draw [] (-1.8,-1.2) to [round right paren ] (-1.8,1.2);
 		\node (B1) at (-.2,0)   {$\bf{=abbb(=ab^3)}$};
 		\fill[black!100!] (-3,-1) circle (.05);
 		\fill[black!100!] (-3,1) circle (.05);
 		\fill[black!100!] (-2,0) circle (.05);
 		\fill[black!100!] (-4,0) circle (.05);
 		\node (B1) at (-3,-1.2)   {$\bf{b}$};
 		\node (B1) at (-3,1.2)   {$\bf{a}$};	
 		\node (B1) at (-4.2,0)   {$\bf{b}$};
 		\node (B1) at (-1.8,0)   {$\bf{b}$};
 		\node (B1) at (-3,-.8)   {$\bf{3}$};
 		\node (B1) at (-3,.8)   {$\bf{1}$};	
 		\node (B1) at (-3.8,0)   {$\bf{4}$};
 		\node (B1) at (-2.2,0)   {$\bf{2}$};
 		%
 		\draw[](4,0) circle (1);
 		\node (B1) at (1,-.1)   {$\bf{,}$};
 		\node (B1) at (2.2,0)   {$\bf{w}$};
 		\draw [] (2.9,-1.2) to [round left paren ] (2.9,1.2);
 		\draw [] (5.2,-1.2) to [round right paren ] (5.2,1.2);
 		\node (B1) at (6.8,0)   {$\bf{=baba(=a^2b^2)}$};
 		\fill[black!100!] (4,1) circle (.05);
 		\fill[black!100!] (5,0) circle (.05);
 		\fill[black!100!] (4,-1) circle (.05);
 		\fill[black!100!] (3,0) circle (.05);
 		\node (B1) at (4,1.2)   {$\bf{b}$};
 		\node (B1) at (5.2, 0)   {$\bf{a}$};	
 		\node (B1) at (4,-1.2)   {$\bf{b}$};
 		\node (B1) at (2.8, 0)   {$\bf{a}$};
 		\node (B1) at (4,.8)   {$\bf{1}$};
 		\node (B1) at (4.8, 0)   {$\bf{2}$};	
 		\node (B1) at (4,-.8)   {$\bf{3}$};
 		\node (B1) at (3.2, 0)   {$\bf{4}$};
 		%
 		\node(B1) at (9.7, -1) {$\bf{etc.}$};
 		\end{tikzpicture}
 		\caption{Two cyclic words with their weights.}
 		\label{weight of word}
 	\end{figure}
 	\begin{figure}[ht!] 
 		\tiny 
 		\tikzstyle{ver}=[]
 		\tikzstyle{vert}=[circle, draw, fill=black!100, inner sep=0pt, minimum width=4pt]
 		\tikzstyle{vertex}=[circle, draw, fill=black!.5, inner sep=0pt, minimum width=4pt]
 		\tikzstyle{edge} = [draw,thick,-]
 		\tikzstyle{node_style} = [circle,draw=blue,fill=blue!20!,font=\sffamily\Large\bfseries]
 		\tikzset{round left paren/.style={ncbar=0.5cm,out=120,in=-120}}
 		\tikzset{round right paren/.style={ncbar=0.5cm,out=60,in=-60}}
 		\begin{tikzpicture}[scale=1]
 		\tikzstyle{edge_style} = [draw=black, line width=2mm, ]
 		\tikzstyle{node_style} = [draw=blue,fill=blue!00!,font=\sffamily\Large\bfseries]
 		\tikzset{
 			LabelStyle/.style = { rectangle, rounded corners, draw,
 				minimum width = 2em, fill = yellow!50,
 				text = red, font = \bfseries },
 			VertexStyle/.append style = { inner sep=5pt,
 				font = \Large\bfseries},
 			EdgeStyle/.append style = {->, bend left} }
 		\tikzset{vertex/.style = {shape=circle,draw,minimum size=1.5em}}
 		\tikzset{edge/.style = {->,> = latex'}}
 		\tikzset{
 			ncbar angle/.initial=90,
 			ncbar/.style={
 				to path=(\tikztostart)
 				-- ($(\tikztostart)!#1!\pgfkeysvalueof{/tikz/ncbar angle}:(\tikztotarget)$)
 				-- ($(\tikztotarget)!($(\tikztostart)!#1!\pgfkeysvalueof{/tikz/ncbar angle}:(\tikztotarget)$)!\pgfkeysvalueof{/tikz/ncbar angle}:(\tikztostart)$)
 				-- (\tikztotarget)
 			},
 			ncbar/.default=0.5cm,
 		}
 		\draw[](-3,3.5) circle (1);
 		\node (B1) at (-4.7,3.5)   {$\bf{w}$};
 		\draw [] (-4.1,2.2) to [round left paren ] (-4.1,4.7);
 		\draw [] (-2,2.2) to [round right paren ] (-2,4.7);
 		\node (B1) at (-.9,3.5)   {$\bf{w}$};
 		\draw [] (-.3,2.2) to [round left paren ] (-.3,4.7);
 		\draw [] (5.1,2.2) to [round right paren ] (5.1,4.7);
 		\fill[black!100!] (-3,4.5) circle (.05);
 		\fill[black!100!] (-2,3.5) circle (.05);
 		\fill[black!100!] (-3,2.5) circle (.05);
 		\fill[black!100!] (-4,3.5) circle (.05);
 		\node (B1) at (-3,4.7)   {$\bf{a+b}$};
 		\node (B1) at (-1.8,3.5)   {$\bf{b}$};	
 		\node (B1) at (-3,2.3)   {$\bf{a}$};
 		\node (B1) at (-4.2,3.5)   {$\bf{b}$};
 		\node (B1) at (-3,4.3)   {$\bf{1}$};
 		\node (B1) at (-2.2,3.5)   {$\bf{2}$};	
 		\node (B1) at (-3,2.7)   {$\bf{3}$};
 		\node (B1) at (-3.8,3.5)   {$\bf{4}$};
 		%
 		\node(B1) at (-1.4, 3.5) {$\bf{=}$};
 		\node(B1) at (-1.4, .5) {$\bf{=}$};
 		\draw[](.7,3.5) circle (1);
 		\fill[black!100!] (.7,4.5) circle (.05);
 		\fill[black!100!] (1.7,3.5) circle (.05);
 		\fill[black!100!] (.7,2.5) circle (.05);
 		\fill[black!100!] (-.3,3.5) circle (.05);
 		\node (B1) at (.7,4.7)   {$\bf{a}$};
 		\node (B1) at (1.9, 3.5)   {$\bf{b}$};	
 		\node (B1) at (.7,2.3)   {$\bf{a}$};
 		\node (B1) at (-.5, 3.5)   {$\bf{b}$};
 		\node (B1) at (.7,4.3)   {$\bf{1}$};
 		\node (B1) at (1.5, 3.5)   {$\bf{2}$};	
 		\node (B1) at (.7,2.7)   {$\bf{3}$};
 		\node (B1) at (-.1,3.5)   {$\bf{4}$};
 		\node(B1) at (2.35, 3.5) {$\bf{+}$};
 		%
 		%
 		%
 		\draw[](4,3.5) circle (1);
 		\fill[black!100!] (4,4.5) circle (.05);
 		\fill[black!100!] (5,3.5) circle (.05);
 		\fill[black!100!] (4,2.5) circle (.05);
 		\fill[black!100!] (3,3.5) circle (.05);
 		\node (B1) at (4,4.7)   {$\bf{b}$};
 		\node (B1) at (5.2, 3.5)   {$\bf{b}$};	
 		\node (B1) at (4,2.3)   {$\bf{a}$};
 		\node (B1) at (2.8, 3.5)   {$\bf{b}$};
 		\node (B1) at (4,4.3)   {$\bf{1}$};
 		\node (B1) at (4.8, 3.5)   {$\bf{2}$};	
 		\node (B1) at (4,2.7)   {$\bf{3}$};
 		\node (B1) at (3.2,3.5)   {$\bf{4}$};
 		\draw[](.7,.5) circle (1);
 		\fill[black!100!] (.7,1.5) circle (.05);
 		\fill[black!100!] (1.7,.5) circle (.05);
 		\fill[black!100!] (.7,-.5) circle (.05);
 		\fill[black!100!] (-.3,.5) circle (.05);
 		\node (B1) at (.7,1.7)   {$\bf{a}$};
 		\node (B1) at (1.9, .5)   {$\bf{b}$};	
 		\node (B1) at (.7,-.8)   {$\bf{a}$};
 		\node (B1) at (-.5, .5)   {$\bf{b}$};
 		\node (B1) at (.7,1.3)   {$\bf{1}$};
 		\node (B1) at (1.5, .5)   {$\bf{2}$};	
 		\node (B1) at (.7,-.3)   {$\bf{3}$};
 		\node (B1) at (-.1,.5)   {$\bf{4}$};
 		\node(B1) at (2.7, .5) {$\bf{+}$};
 		\node(B1) at (3.3, .5) {$\bf{w}$};
 		%
 		%
 		%
 		\draw[](5,.5) circle (1);
 		\fill[black!100!] (5,1.5) circle (.05);
 		\fill[black!100!] (6,.5) circle (.05);
 		\fill[black!100!] (5,-.5) circle (.05);
 		\fill[black!100!] (4,.5) circle (.05);
 		\node (B1) at (5,1.7)   {$\bf{b}$};
 		\node (B1) at (6.2, .5)   {$\bf{b}$};	
 		\node (B1) at (5,-.8)   {$\bf{a}$};
 		\node (B1) at (3.8, .5)   {$\bf{b}$};
 		\node (B1) at (5,1.3)   {$\bf{1}$};
 		\node (B1) at (5.8, .5)   {$\bf{2}$};	
 		\node (B1) at (5,-.3)   {$\bf{3}$};
 		\node (B1) at (4.2,.5)   {$\bf{4}$};
 		\node (B1) at (-.9,.5)   {$\bf{w}$};
 		\draw [] (-.3,-.7) to [round left paren ] (-.3,1.7);
 		\draw [] (1.7,-.7) to [round right paren ] (1.7,1.7);
 		\draw [] (3.9,-.7) to [round left paren ] (3.9,1.7);
 		\draw [] (6.1,-.7) to [round right paren ] (6.1,1.7);
 		\node(B1) at (1.8, -1.3) {$\bf{=} abab(=a^2b^2)+bbab(=b^3a)=a^2b^2+ab^3$};
 		\end{tikzpicture}
 		\caption{ This figure describes the weights of the sum of cyclic words}
 		\label{additivity of weight}	
 	\end{figure}
 	Clearly, Figure \ref{fig:f3} shows the required bijection. Also this bijection is sign and weight preserving. Again $(-1)^{4-1}w(L_1)=a^4$ and $(-1)^{4-1}w(L_2)=b^4.$  Hence $\det(S)=2(a^4+b^4).$ 
 	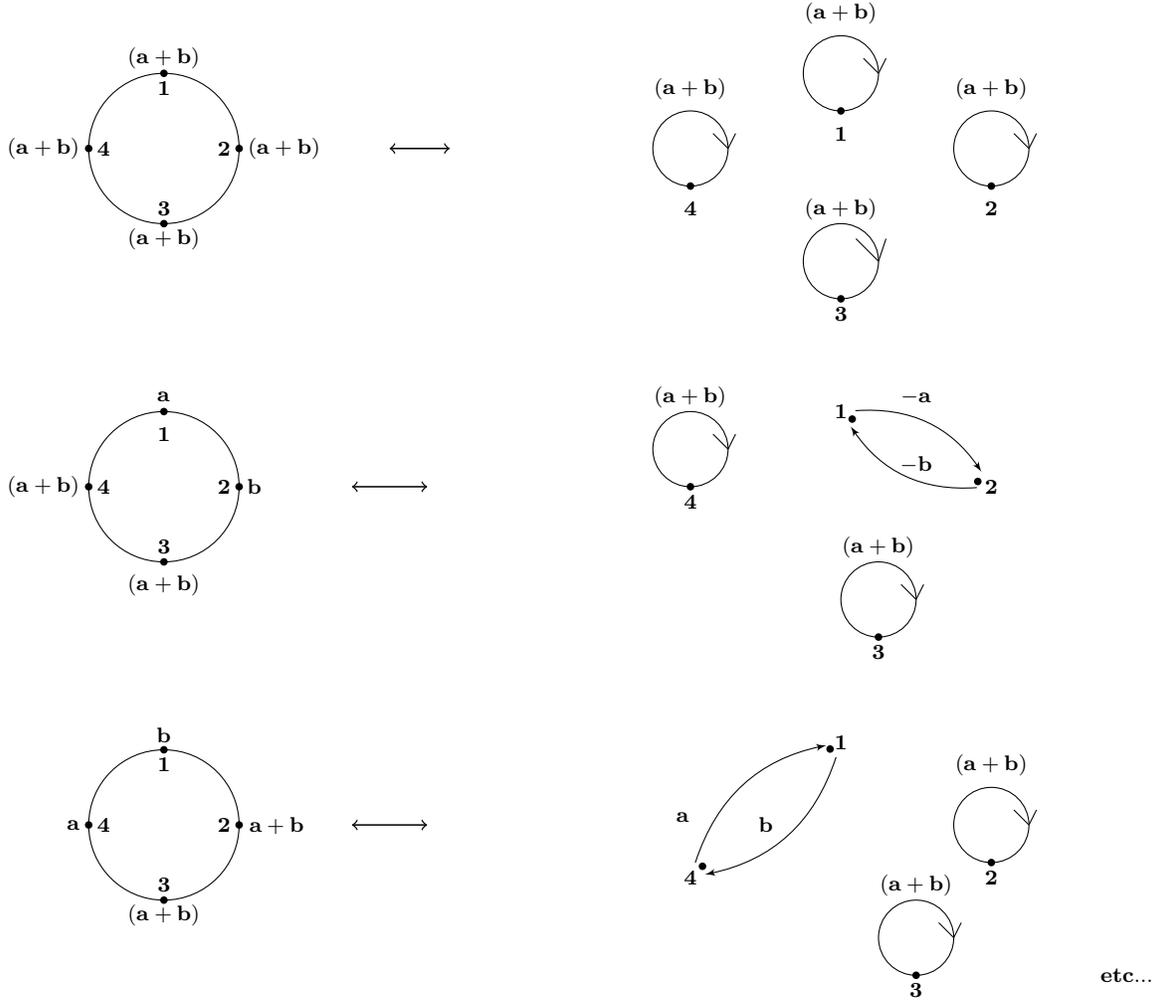
\begin{figure}[ht!]
 		\tiny
 		\tikzstyle{ver}=[]
 		\tikzstyle{vert}=[circle, draw, fill=black!100, inner sep=0pt, minimum width=4pt]
 		\tikzstyle{vertex}=[circle, draw, fill=black!.5, inner sep=0pt, minimum width=4pt]
 		\tikzstyle{edge} = [draw,thick,-]
 		\tikzstyle{node_style} = [circle,draw=blue,fill=blue!20!,font=\sffamily\Large\bfseries]
 		\tikzset{->-/.style={decoration={
 					markings,
 					mark=at position #1 with {\arrow{>}}},postaction={decorate}}}
 		\begin{tikzpicture}[scale=1]
 		\tikzstyle{edge_style} = [draw=black, line width=2mm, ]
 		\tikzstyle{node_style} = [draw=blue,fill=blue!00!,font=\sffamily\Large\bfseries]
 		\tikzset{
 			LabelStyle/.style = { rectangle, rounded corners, draw,
 				minimum width = 2em, fill = yellow!50,
 				text = red, font = \bfseries },
 			VertexStyle/.append style = { inner sep=5pt,
 				font = \Large\bfseries},
 			EdgeStyle/.append style = {->, bend left} }
 		\tikzset{vertex/.style = {shape=circle,draw,minimum size=1.5em}}
 		\tikzset{edge/.style = {->,> = latex'}}
 		\draw[](-3,0) circle (1);
 		\fill[black!100!] (-3,1) circle (.05);
 		\fill[black!100!] (-2,0) circle (.05);
 		\fill[black!100!] (-3,-1) circle (.05);
 		\fill[black!100!] (-4,0) circle (.05);
 		\node (B1) at (-3,1.2)   {$\bf{(a+b)}$};
 		\node (B1) at (-1.4,0)   {$\bf{(a+b)}$};	
 		\node (B1) at (-3,-1.2)   {$\bf{(a+b)}$};
 		\node (B1) at (-4.6,0)   {$\bf{(a+b)}$};
 		\node (B1) at (-3,.8)   {$\bf{1}$};
 		\node (B1) at (-2.2,0)   {$\bf{2}$};	
 		\node (B1) at (-3,-.8)   {$\bf{3}$};
 		\node (B1) at (-3.8,0)   {$\bf{4}$};
 		\draw[<->, line width=.2 mm] (0,0) -- (.8,0);

 		\draw[](6,1) circle (.5);
 		\draw[] (6.5,1) -- (6.3,1.2);
 		\draw[] (6.5,1) -- (6.6,1.2);
 		\node (B1) at (6,.2)   {$\bf{1}$};
 		\fill[black!100!](6,.5) circle (.05);
 	\node (B1) at (6,1.8)   {$\bf{(a+b)}$};
 		\draw[](6,-1.5) circle (.5);
 		\draw[] (6.5,-1.5) -- (6.2,-1.2);
 		\draw[] (6.5,-1.5) -- (6.6,-1.2);
 		\node (B1) at (6,-2.2)   {$\bf{3}$};
 		\fill[black!100!](6,-2) circle (.05);
 		\node (B1) at (6,-.8)   {$\bf{(a+b)}$};
 	\draw[](8,0) circle (.5);
 	\draw[] (8.5,0) -- (8.3,.2);
 	\draw[] (8.5,0) -- (8.6,.2);
 	\node (B1) at (8,-.8)   {$\bf{2}$};
 	\fill[black!100!](8,-.5) circle (.05);
 	\node (B1) at (8,.8)   {$\bf{(a+b)}$};
 		\draw[](4,0) circle (.5);
 			\draw[] (4.5,0) -- (4.3,.2);
 			\draw[] (4.5,0) -- (4.6,.2);
 			\node (B1) at (4,-.8)   {$\bf{4}$};
 			\fill[black!100!](4,-.5) circle (.05);
 			\node (B1) at (4,.8)   {$\bf{(a+b)}$};
 			%
 		%
 		\draw[](-3,-4.5) circle (1);
 		\fill[black!100!] (-3,-3.5) circle (.05);
 		\fill[black!100!] (-4,-4.5) circle (.05);
 		\fill[black!100!] (-3,-5.5) circle (.05);
 		\fill[black!100!] (-2,-4.5) circle (.05);
 		\node (B1) at (-3,-3.3)   {$\bf{a}$};
 		\node (B1) at (-1.8,-4.5)   {$\bf{b}$};	
 		\node (B1) at (-3,-5.8)   {$\bf{(a+b)}$};
 		\node (B1) at (-4.6,-4.5)   {$\bf{(a+b)}$};
 		\node (B1) at (-3,-3.8)   {$\bf{1}$};
 		\node (B1) at (-2.2,-4.5)   {$\bf{2}$};	
 		\node (B1) at (-3,-5.3)   {$\bf{3}$};
 		\node (B1) at (-3.8,-4.5)   {$\bf{4}$};
 		\draw[<->, line width=.2 mm] (-.5,-4.5) -- (.5,-4.5);
 		\node[] (a) at  (6,-3.5) {$\bf{1}$};
 		\node[] (b) at  (8,-4.5) {$\bf{2}$};
 		\node[] (c) at  (4,-4.7) {$\bf{4}$};
 		\draw[edge] (a)  to[bend left] (b);
 		\draw[edge] (b)  to[bend left] (a);
 		\fill[black!100!](6.15,-3.6) circle (.05);
 		\fill[black!100!](7.82,-4.43) circle (.05);
 		\node (B1) at (7,-3.3)   {$\bf{-a}$};
 		\node (B1) at (7,-4.2)   {$\bf{-b}$};
\draw[](6.5,-6) circle (.5);
 			\draw[] (7,-6) -- (7.1,-5.8);
 			\draw[] (7,-6) -- (6.8,-5.8);
 			\node (B1) at (6.5,-6.7)   {$\bf{3}$};
 			\fill[black!100!](6.5,-6.5) circle (.05);
 			\node (B1) at (6.5,-5.3)   {$\bf{(a+b)}$};
 			\draw[](4,-4) circle (.5);
 		\draw[] (4.5,-4) -- (4.6,-3.8);
 		\draw[] (4.5,-4) -- (4.3,-3.8);
 			%
 		\fill[black!100!](4, -4.5) circle (.05);
 			\node (B1) at (4,-3.3)   {$\bf{(a+b)}$};
 	\draw[](-3,-9) circle (1);
 		\fill[black!100!] (-3,-8) circle (.05);
 		\fill[black!100!] (-2,-9) circle (.05);
 		\fill[black!100!] (-3,-10) circle (.05);
 		\fill[black!100!] (-4,-9) circle (.05);
 		\node (B1) at (-3,-7.8)   {$\bf{b}$};
 		\node (B1) at (-1.5,-9)   {$\bf{a+b}$};	
 		\node (B1) at (-3,-10.2)   {$\bf{(a+b)}$};
 		\node (B1) at (-4.2,-9)   {$\bf{a}$};
 		\node (B1) at (-3,-8.2)   {$\bf{1}$};
 		\node (B1) at (-2.2,-9)   {$\bf{2}$};	
 		\node (B1) at (-3,-9.8)   {$\bf{3}$};
 		\node (B1) at (-3.8,-9)   {$\bf{4}$};
 	\draw[<->, line width=.2 mm] (-.5,-9) -- (.5,-9);
 		%
 		\fill[black!100!](5.856,-7.99) circle (.05);
 		\fill[black!100!](4.16,-9.55) circle (.05);
 		\node[] (e) at  (6,-7.9) {$\bf{1}$};
 		\node[] (f) at  (4,-9.7) {$\bf{4}$};
 			\draw[edge] (e)  to[bend left] (f);
 			\draw[edge] (f)  to[bend left] (e);
 		\node[] (c) at  (5,-8.99) {$\bf{b}$};
 		\node[] (c) at  (3.9,-8.9) {$\bf{a}$};
 		%
 		%
 		\draw[](7,-10.5) circle (.5);
 		\draw[] (7.5,-10.5) -- (7.6,-10.3);
 		\draw[] (7.5,-10.5) -- (7.3,-10.3);
 			\node[] (e) at  (7,-11.2) {$\bf{3}$};
 				\fill[black!100!](7,-11) circle (.05);
 				\node (B1) at (7,-9.8)   {$\bf{(a+b)}$};
 		\draw[](8,-9) circle (.5);
 		\draw[] (8.5,-9) -- (8.3,-8.8);
 		\draw[] (8.5,-9) -- (8.6,-8.8);
 		\node (B1) at (8,-9.7)   {$\bf{2}$};
 		\fill[black!100!](8,-9.5) circle (.05);
 		\node (B1) at (8,-8.2)   {$\bf{(a+b)}$};
 \node (B1) at (9.8, -11)   {$\bf{etc...}$};
 		\end{tikzpicture}
 		\caption{The numbers appearing  on each edge in the above figure is the weight of the corresponding edge. The left hand side of this figure describes some terms of PIE expression.}
 		\label{fig:f3}	
 	\end{figure}
\end{proof}
Now, we show a determinantal expression of the Lucas numbers. First we consider an $n\times n (n\geq 3)$ matrix \[A=\left(
 \begin{array}{cccccc}
 1 & (-1)^{n+1}\frac{1+\sqrt{5}}{2} & 0  &\cdots  & 0 & \frac{1-\sqrt{5}}{2}\\
 (-1)^{n+1}\frac{1-\sqrt{5}}{2} & 1 & \frac{1+\sqrt{5}}{2}   &\cdots & 0 & 0 \\
 0 & \frac{1-\sqrt{5}}{2}  & 1 & \cdots & 0  & 0  \\
 \vdots & \vdots & \vdots & \ddots & \vdots & \vdots \\
 0 & 0 & 0 & \cdots  & 1 & \frac{1+\sqrt{5}}{2}  \\
 \frac{1+\sqrt{5}}{2} & 0 & 0 & \cdots & \frac{1-\sqrt{5}}{2} & 1 \\
 \end{array}
 \right).\]
 \begin{corollary}
Let $\ell_n (n\geq 3)$ be the $n$-th term of the Lucas number. Then $\ell_n=\frac{1}{2}\det(A).$ 	
 \end{corollary}
 \begin{proof}
If we put $a=\frac{1+\sqrt{5}}{2}$ and $b=\frac{1-\sqrt{5}}{2}$ in the matrix $S,$ then we get the matrix $A.$ So, \[\det(A)=2\left[\left(\frac{1+\sqrt{5}}{2} \right)^n+\left(  \frac{1-\sqrt{5}}{2} \right)^n\right].\] Now using the proof of Theorem \ref{gen iden lucas} and combinatorial interpretation of the Lucas numbers, we can write \[\det(A)=\ell_n+\left(\frac{1+\sqrt{5}}{2} \right)^n+\left(  \frac{1-\sqrt{5}}{2} \right)^n.\] Hence the corollary.
 \end{proof}
 \subsection*{Acknowledgement} I would like to thank my mentor Prof. Arvind Ayyer  for his constant support, encouragement and for valuable discussion and  suggestions in the preparation of this paper. Also I would like to thank Dr. Sajal Kumar Mukherjee for many helpful discussion and proposing Theorem $4.1.$ The author was supported by Department of Science and Technology grant EMR/2016/006624 and partly supported by  UGC Centre for Advanced Studies. Also the author was supported by NBHM Post Doctoral Fellowship grant 0204/52/2019/RD-II/339. 
\bibliographystyle{amsplain}
\bibliography{gen-inv-lcp}

\providecommand{\bysame}{\leavevmode\hbox to3em{\hrulefill}\thinspace}
\providecommand{\MR}{\relax\ifhmode\unskip\space\fi MR }
\providecommand{\MRhref}[2]{%
  \href{http://www.ams.org/mathscinet-getitem?mr=#1}{#2}
}
\providecommand{\href}[2]{#2}
\begin{thebibliography}{10}

\bibitem{18}
A.~Ayyer, \emph{Determinants and perfect matchings}, J. Combin. Theory Ser. A
  \textbf{120} (2013), no.~1, 304--314.

\bibitem{25}
A.~T. Benjamin, H.~Derks, and J.~J. Quinn, \emph{The combinatorialization of
  linear recurrences}, Electron. J. Combin. \textbf{18} (2011), no.~2, Paper
  12, 18.

\bibitem{29}
A.~T. Benjamin, G.~M. Levin, K.~Mahlburg, and J.~J. Quinn, \emph{Random
  approaches to {F}ibonacci identities}, Amer. Math. Monthly \textbf{107}
  (2000), no.~6, 511--516.

\bibitem{24}
A.~T. Benjamin and J.~J. Quinn, \emph{Proofs that really count: The art of
  combinatorial proof}, Dolciani Series, Mathematical Association of America,
  Washington DC, 2003.

\bibitem{14}
S.~Bera and S.~K. Mukherjee, \emph{Combinatorial proofs of some determinantal
  identities}, Linear and Multilinear Algebra \textbf{66} (2018), no.~8,
  1659--1667.

\bibitem{16}
A.~R. Brualdi and D.~Cvetkovic, \emph{A combinatorial approach to matrix theory
  and its application}, Discrete Mathematics and Its Applications, vol.~52, CRC
  Press, Boca Raton, London, New York, 2009.

\bibitem{26}
J.~McLaughlin and B.~Sury, \emph{Powers of a matrix and combinatorial
  identities}, Integers \textbf{5} (2005), no.~1, A13, 9.

\bibitem{SSDISC}
S.~K. Mukherjee and S.~Bera, \emph{Combinatorial proofs of the
  {N}ewton--{G}irard and {C}hapman--{C}ostas-{S}antos identities}, Discrete
  Math. \textbf{342} (2019), no.~6, 1577--1580.

\bibitem{21}
B.~E. Sagan, \emph{The symmetric group representations, combinatorial
  algorithms, and symmetric functions}, Graduate Texts in Mathematics, vol.
  238, Springer-Verlag New York, Inc., 2001.

\bibitem{27}
B.~Sury, \emph{A curious polynomial identity}, Nieuw Arch. Wisk. (4)
  \textbf{11} (1993), no.~2, 93--96.

\bibitem{11}
D.~Zeilberger, \emph{A combinatorial proof of newton's identity}, Discrete
  Mathematics \textbf{49} (1984), 319.

\end{thebibliography}
\end{document}